\documentclass[12pt,reqno]{amsart}
\usepackage{amsmath,amssymb,mathrsfs}

\topmargin by -\headsep

\newtheorem{theorem}{Theorem}[section]
\newtheorem{lemma}[theorem]{Lemma}

\theoremstyle{definition}

\theoremstyle{remark}

\numberwithin{equation}{section}
\newcommand{\mmod}[1]{\,\,(\text{mod}\,\,#1)}

\def\g{{\mathbf g}}
\def\h{{\mathbf h}}

\def\v{{\mathbf v}}

\def\x{{\mathbf x}}
\def\y{{\mathbf y}}

\def\d{{\,{\rm d}}}
\def\ftil{{\widetilde f}}

\def\L{{\boldsymbol \lambda}}

\def\cA{{\mathcal A}}

\def\cC{{\mathcal C}}

\def\cG{{\mathcal G}}

\def\cI{{\mathcal I}}
\def\cJ{{\mathcal J}}

\def\cM{{\mathcal M}}
\def\cN{{\mathcal N}}

\def\cR{{\mathcal R}}

\def\cV{{\mathcal V}}
\def\cW{{\mathcal W}}
\def\cX{{\mathcal X}}
\def\cY{{\mathcal Y}}
\def\cZ{{\mathcal Z}}

\def\tg{\widetilde{g}}\def\tfh{\widetilde{\mathfrak h}}

\def\tf{\widetilde{f}}\def\tZ{\widetilde{Z}}

\def\R{{\mathbb R}}\def\NN{{\mathbb N}}\def\C{{\mathbb C}}
\def\Z{{\mathbb Z}}\def\Q{{\mathbb Q}}

\def\fJ{{\mathfrak J}}\def\ff{{\mathfrak f}}\def\fg{{\mathfrak g}}\def\fh{{\mathfrak h}}
\def\tff{{\widetilde \ff}}
\def\fm{{\mathfrak m}}\def\fM{{\mathfrak M}}\def\fN{{\mathfrak N}}
\def\fn{{\mathfrak n}}\def\fp{{\mathfrak p}}
\def\fq{{\mathfrak q}}  \def\fs{{\mathfrak s}}
\def\fw{{\mathfrak w}}\def\fW{{\mathfrak W}}\def\fB{{\mathfrak B}}\def\ft{{\mathfrak t}}
\def\hK{{\widehat{K}}}\def\ctZ{{\widetilde{\mathcal
Z}}}\def\fC{{\mathfrak C}}
\def\fv{{\mathfrak v}}\def\fV{{\mathfrak V}}

\def\ep{\varepsilon}
\def\implies{\Rightarrow}

\def\le{\leqslant}
\def\ge{\geqslant}
\begin{document}

\title{Exceptional sets for Diophantine inequalities}

\author[Scott T. Parsell and Trevor D. Wooley]{Scott T. Parsell and Trevor D. Wooley}
\address{STP: Department of Mathematics,
West Chester University, 25 University Ave., West Chester, PA 19383, U.S.A.}
\email{sparsell@wcupa.edu}
\address{TDW: Department of Mathematics, University of Bristol,
University Walk, Clifton, Bristol BS8 1TW, United Kingdom}
\email{matdw@bristol.ac.uk}
\thanks{STP was supported by National Security Agency Grants H98230-08-1-0070 and H98230-11-1-0190 and TDW by a Royal Society Wolfson Research Merit Award.}

\begin{abstract} We apply Freeman's variant of the Davenport-Heilbronn method to investigate the exceptional set of real numbers not close to some value of a given real diagonal form at an integral argument. Under appropriate conditions, we show that the exceptional set in the interval $[-N,N]$ has measure $O(N^{1-\delta})$, for a positive number $\delta$.
\end{abstract}

 \maketitle

\section{Introduction}\label{sect:1} A variant of the classical circle method introduced by Davenport and Heilbronn \cite{DH:46} permits the investigation of the value distribution of indefinite real diagonal forms at integral points. Let $k\in \NN$ and let $\lambda_1, \dots, \lambda_s$ be non-zero real numbers, not all in rational ratio, and not all of the same sign when $k$ is even. Then the Davenport-Heilbronn method establishes the existence of a natural number $\fs(k)$ having the property that, whenever $s\ge \fs(k)$, then for all $0<\tau\le 1$ and $\mu\in \R$, there exist infinitely many integral solutions $\x$ of the inequality
\begin{equation}\label{1.1}
|\lambda_1x_1^k+\dots +\lambda_sx_s^k-\mu|<\tau.
\end{equation}
When $\mu=0$ and $\fs(k)=2^k+1$, a result of this type is described, for example, in \cite[Theorem 11.1]{V:HL}. For smaller values of $s$, available technology may limit accessible conclusions to analogues involving some sort of averaging over the real number $\mu$. Quantitative estimates in this direction are ultimately connected to the savings achievable on a suitably defined set of minor arcs, and so even the pivotal reorganization of the Davenport-Heilbronn method introduced by Freeman \cite{Freem:alb} apparently limits such bounds to be better than trivial only by the narrowest of margins. In this paper we demonstrate that, in a wide set of circumstances, this barrier may be unequivocally broken. Indeed, we show that the Diophantine inequality (\ref{1.1}) is satisfied for all real numbers $\mu \in [-N,N]$ with the possible exception of a set having measure $O(N^{1-\Delta})$, for a positive number $\Delta$.\par

In order to advance further, we must formalise the above ideas and introduce notation with which to describe our conclusions. Let $\cZ_{s,k}(N,M)=\cZ_{s,k}^\tau (N,M;\L)$ denote the set of real numbers $\mu \in [N,N+M]$ for which the inequality (\ref{1.1}) has no integral solution. On writing
\begin{equation}\label{1.2}
F(\x)=\lambda_1 x_1^k + \dots + \lambda_sx_s^k
\end{equation}
and
$$\fB= \bigcap_{\x \in \Z^s} (-\infty, F(\x)-\tau] \cup [F(\x)+\tau,\infty),$$
one sees that $\cZ_{s,k}(N,M)=\fB \cap[N,N+M]$. This shows that $\cZ_{s,k}(N,M)$ is a closed subset of $\R$ and in particular that it is measurable. We may therefore write
$$Z_{s,k}(N,M)={\rm{meas}}(\cZ_{s,k}(N,M)).$$
Although the solubility of (\ref{1.1}) for all $\mu$ requires an indefiniteness hypothesis, our results on exceptional sets apply equally well to situations in which the form $F(\x)$ is definite. The proofs do require slightly more care in the definite case to ensure compatibility between the size of $\mu$ and the ranges of the variables, but this is easily arranged by summing over dyadic intervals.  We concentrate on the case where $F$ has at least one positive coefficient exceeding $2$ and leave to the reader the necessary sign changes and re-scaling required to formulate the general case.\par

Our results are concisely introduced in some generality by reference to available minor arc estimates for exponential sums. With this end in mind we record some additional notation. When $P$ and $R$ are positive real numbers, write
$$\cA(P,R)=\{n\in [1,P]\cap\Z:\text{$p|n \implies p \le R$}\}.$$
Here and throughout, we employ the letter $p$ to denote a prime number, and we assume that $R$ is a sufficiently small positive power of $P$. In addition, we put
\begin{equation}\label{1.3}
g(\alpha)=\sum_{x \in \cA(P,R)} e(\alpha x^k),
\end{equation}
where as usual $e(z)$ denotes $e^{2\pi iz}$. We next define a standard set of major and minor arcs. When $q\in \NN$ and $a\in \Z$, we take $\fN(q,a)$ to be the set of real numbers $\alpha$ for which $|q\alpha-a| \le (2k)^{-1}P^{1-k}$.  We write $\fN$ for the union of the intervals $\fN(q,a)$ over all co-prime integers $a$ and $q$ with $1 \le q \le (2k)^{-1}P$, and we put $\fn=\R \setminus \fN$. Finally, we say that the pair $(s,\sigma)$ forms a {\it smooth accessible pair for $k$} when $s\ge 2k$, and for some positive number $\omega$, whenever $\ep>0$, then
\begin{equation}\label{1.4}
\int_{\fn \cap [0,1)}|g(\alpha)|^s\d\alpha \ll P^{s-k-\omega}\quad \mbox{and} \quad \sup_{\alpha \in \fn} |g(\alpha)| \ll P^{1-\sigma+\ep}.
\end{equation}

In \S\ref{sect:6} we establish a bound on the measure of the exceptional set $\cZ_{s,k}(N,M)$. Here and elsewhere, we adopt the convention that whenever $\ep$ appears in a statement, then we implicitly assert that the statement holds for all $\ep>0$. Moreover, implicit constants in Vinogradov's notation may depend on $s$, $k$, $\L$ and $\tau$.

\begin{theorem}\label{theorem1.1}
Suppose that $k\ge 3$ and that $(s_0,\sigma_0)$ forms a smooth accessible pair for $k$. Let $s$ and $t$ be non-negative integers with
$$s\ge \max\{2k+3,17,\tfrac{1}{2}s_0\}.$$
Then whenever $M \ge N^{(1-1/k)^t}$, there exists a positive number $\Delta$ such that
\begin{equation}\label{1.5}
Z_{s+t,k}(N,M)\ll MN^{-\Delta-(2s-s_0)\sigma_0/k}.
\end{equation}
\end{theorem}

Subject to the hypotheses of the statement of this theorem, the methods discussed in \cite{W:FAF} permit the proof of the bound $Z_{s,k}(N,M)\ll 1$ when $s\ge s_0$. Indeed, when the form $\lambda_1x_1^k+\ldots +\lambda_sx_s^k$ is indefinite the latter bound may be replaced by the definitive statement that $Z_{s,k}(N,M)=0$. In \S\S\ref{sect:6} and \ref{sect:7} we apply the work of \cite{BW:00-27, VW:95, VW:00, W:92war, W:95sws} to provide a refined explicit version of Theorem \ref{theorem1.1}.

\begin{theorem}\label{theorem1.2}
Suppose that $4\le k\le 20$, and that $s_0=s_0(k)$, $u_0=u_0(k)$ and $\sigma=\sigma(k)$ are as given in Table $1$. Then whenever $s$ and $t$ are non-negative integers with $s\ge \frac{1}{2}(s_0+u_0)$, and $M \ge N^{(1-1/k)^t}$, there exists a positive number $\Delta$ such that
\begin{equation*}
Z_{s+t,k}(N,M) \ll MN^{-\Delta-(2s-s_0)\sigma/k}.
\end{equation*}
Moreover, when $s\ge \frac{1}{2}s_0$, and $M \ge N^{(1-1/k)^t}$, then for some $\Delta>0$ one has
\begin{equation*}
Z_{s+t,k}(N,M) \ll MN^{-\Delta}.
\end{equation*}
The same conclusions hold for larger values of $k$ on setting $u_0(k)=0$,
$$s_0(k) = k(\log k + \log \log k + 2+o(1))\quad \text{and}\quad \sigma(k)^{-1}=k(\log k+O(\log \log k)).$$
\end{theorem}

$$\boxed{\begin{matrix} k&4&5&6&7&8&9&10&11&12\\
s_0(k)&12&18&25&33&42&50&59&67&76\\
u_0(k)&4&6&7&0&0&0&0&0&0\\
\sigma(k)^{-1}&8&16&32&58&70&83&95&108&120\end{matrix}}$$
$$\boxed{\begin{matrix}k&13&14&15&16&17&18&19&20\\
s_0(k)&84&92&100&109&117&125&134&142\\
u_0(k)&0&0&0&0&0&0&0&0\\
\sigma(k)^{-1}&133&146&158&171&184&197&210&223\end{matrix}}$$
\vskip.2cm
\begin{center}\text{Table 1: Parameters for Theorem \ref{theorem1.2}}\end{center}
\vskip.1cm
\noindent Write $\cZ_{s,k}^\tau (N;\L)$ for the set of real numbers $\mu \in [-N,N]$ for which the Diophantine inequality (\ref{1.1}) has no solution, and let $Z_{s,k}(N)$ denote its measure. Then by applying Theorem \ref{theorem1.1} with $t=0$ and summing over dyadic intervals, one finds that there is a positive number $\Delta$ such that, whenever
$$s\ge \max \{2k+3,17,\tfrac{1}{2}s_0(k)\},$$
then $Z_{s,k}(N)=O(N^{1-\Delta})$. This confirms the estimate advertised in our opening paragraph. Furthermore, if $t \sim \beta k \log k$, where $0 < \beta \le \frac{1}{2}$, then one finds that Theorem \ref{theorem1.1} applies with intervals of length $M \asymp N^{\gamma}$, where $\gamma=k^{-\beta+o(1)}$. Hence if $s_0(k) \ge k\log k$, then there are integers $s < s_0(k)$ for which the estimate (\ref{1.5}) holds with $M$ of approximate order $N^{1/\sqrt{k}}$.\par

In the special case $k=3$, additional control may be exercised over exponential sum estimates, and this permits several refinements over the conclusion of Theorem \ref{theorem1.1}. We illustrate such ideas in \S\ref{sect:8} by establishing the following theorem.

\begin{theorem}\label{theorem1.3}
Let $\xi$ be a real number satisfying $\xi^{-1}>2556+48\sqrt{2833}$. Then
$$Z_{4,3}(N)=o(N),\quad Z_{5,3}(N)\ll N^{3/4-\xi},\quad \text{and}\quad Z_{6,3}(N)\ll N^{1/2-2\xi}.$$
\end{theorem}

Analogous conclusions for Waring's problem are sharper, most notably for sums of four cubes (see \cite{Bru:91, KW:10rel, W:95bcc, W:3cubes}). Experts will find the explanation in the absence of a $p$-adic iteration restricted to minor arcs in the context of Diophantine inequalities.\par

We now consider asymptotic formulae subject to the restriction that $0<\tau\le 1$. Denote by $\cN_{s,k}^\tau(P;\L,\mu)$ the number of integral solutions of the inequality (\ref{1.1}) with $\x\in [1,P]^s$, and note our earlier assumption that no coefficient $\lambda_i$ is zero. We put $\nu_i=|\lambda_i|$ and $\sigma_i=\lambda_i/\nu_i$, and then define $\fV(\theta) $ to be the set of $\v\in [0,\nu_1]\times \dots \times [0,\nu_{s-1}]$ satisfying the condition
$$\sigma_s(\theta -\sigma_1v_1-\dots -\sigma_{s-1}v_{s-1})\in [0,\nu_s].$$
Finally, we write
$$\Omega_{s,k}(\L,\theta)=k^{-s}|\lambda_1\dots \lambda_s|^{-1/k}C_{s,k}(\L,\theta),$$
where
$$C_{s,k}(\L,\theta)=\int_{\fV(\theta)}(v_1\dots v_{s-1})^{1/k-1}(\sigma_s(\theta -\sigma_1v_1-\dots
 -\sigma_{s-1}v_{s-1}))^{1/k-1}\d\v .$$
When $k\ge 3$ and $s\ge k+1$, it follows from a heuristic application of the Davenport-Heilbronn method that there is a function $L(P)$ tending to infinity such that
\begin{equation}\label{1.6}
\cN_{s,k}^\tau (P;\L,\mu) = 2\tau\Omega_{s,k}(\L,\mu P^{-k})P^{s-k} +O(P^{s-k}L(P)^{-1}).
\end{equation}
We note that $1\ll \Omega_{s,k}(\L,\theta)\ll 1$ provided only that $\text{meas}(\fV(\theta))\gg 1$, and in such circumstances the relation (\ref{1.6}) constitutes an honest asymptotic formula.\par

Next, let $z$ be a positive parameter, and consider a positive function $\psi(z)$ growing sufficiently slowly in terms of $z$. We denote by $\ctZ_{s,k}^\tau(N;\psi;\L)$ the set of real numbers $\mu \in (N/2,N]$ for which one has
\begin{equation}\label{1.7}
|\cN_{s,k}^{\tau}(N^{1/k};\L,\mu) -2\tau\Omega_{s,k}(\L,\mu/N)N^{s/k-1}|>N^{s/k-1}\psi(N)^{-1}.\end{equation}
Recall the notation introduced in (\ref{1.2}) and write $\chi(\mu;a,b)$ for the indicator function of the interval $(a,b)$, so that
$$\chi(\mu;a,b)=\begin{cases} 1,&\text{when $a<\mu<b$,}\\
0,&\text{when $\mu\le a$ or $\mu\ge b$.}\end{cases}$$
Then we see that the counting function $\cN_{s,k}^\tau(\mu)=\cN_{s,k}^\tau(N^{1/k};\L,\mu)$ can be defined by means of the relation
$$\cN_{s,k}^\tau(\mu) = \sum_{1\le x_1,\ldots ,x_s\le N^{1/k}}\chi(\mu;F(\x)-\tau,F(\x)+\tau).$$
Our earlier assumption that one at least of the coefficients $\lambda_i$ is positive and exceeds $2$ implies that  $\text{meas}(\fV(\mu/N))>0$, and hence $\Omega=\Omega_{s,k}(\L,\mu/N)>0$. We therefore find that $\cN_{s,k}^\tau (\mu)-2\tau \Omega N^{s/k-1}$ is a measurable function of $\mu$, and hence the set $\ctZ_{s,k}^\tau (N;\psi;\L)$ is measurable. We write
$$\tZ_{s,k}(N) = {\rm{meas}}(\ctZ_{s,k}^\tau (N;\psi;\L)).$$

\par As in our earlier discussion of the counting function $Z_{s,k}(N)$, we introduce some notation with which to discuss minor arc estimates for classical Weyl sums. Write
$$f(\alpha)=\sum_{1 \le x \le P}e(\alpha x^k).$$
We say that the triple $(s,\sigma,U)$ forms an {\it accessible triple for $k$} when $s\ge 2k$, the function $U(P)$ increases monotonically to infinity, and one has
\begin{equation}\label{1.8}
\int_{\fn \cap [0,1)}|f(\alpha)|^s\d\alpha \ll P^{s-k}U(P)^{-1}\quad \text{and}\quad \sup_{\alpha \in \fn} |f(\alpha)| \ll P^{1-\sigma+\ep}.
\end{equation}

In \S\ref{sect:3} we investigate the measure of the exceptional set $\tZ_{s,k}(N)$.

\begin{theorem}\label{theorem1.4}
Suppose that $k \ge 3$ and that $(s_1,\sigma_1,U)$ forms an accessible triple for $k$. Then whenever
$$s\ge \max\{k+2,\tfrac{1}{2}s_1\},\quad \sigma<\sigma_1,$$
and $\psi(N)$ grows sufficiently slowly in terms of $s$, $\sigma$, $k$, $\L$, and $\tau$, one has
$$\tZ_{s,k}(N)\ll N^{1-(2s-s_1)\sigma/k}U(N^{1/k})^{\ep-1}.$$
\end{theorem}

Although Theorem \ref{theorem1.4} does not address estimates for $\tZ_{4,3}(N)$, a fairly pedestrian approach yields the estimate contained in the following theorem.

\begin{theorem}\label{theorem1.5}
Whenever $\psi(N)$ grows sufficiently slowly in terms of $\L$ and $\tau$, one has $\tZ_{4,3}(N)=o(N)$.
\end{theorem}

Subject to the hypotheses of the statement of Theorem \ref{theorem1.4}, the methods discussed in \cite{W:FAF} may be applied to show that $\tZ_{s,k}(N)\ll 1$ when $s\ge s_1$. The best accessible triples for smaller values of $k$ stem from work of Boklan \cite{Bok:93} and Vaughan \cite{Va:86c}, \cite{Va:86se}. Thus, for a suitable positive number $\gamma$, one has the accessible triples
$$(s_1,\sigma_1,U(N))=(2^k,2^{1-k},(\log N)^\gamma)\quad (3\le k\le 5).$$
For larger values of $k$ the picture has recently been transformed by developments stemming from the second author's efficient congruencing approach to Vinogradov's mean value theorem (see \cite{W:12war}, \cite{W:12ec}, \cite{W:13ec2}). In particular, on defining the exponent $\sigma^*(k)$ by means of the relation
$$1/\sigma^*(k)=\begin{cases}2^{k-1},&\text{when $k=6,7$,}\\
2k(k-2),&\text{when $k\ge 8$,}
\end{cases}$$
it follows that for a suitable positive number $\gamma$ one has the accessible triples
$$(s_1,\sigma_1,U(N))=(2k^2-2k-8,\sigma^*(k),N^\gamma)\quad (k \ge 6)$$
and
$$(s_1,\sigma_1,U(N))=(2k^2-2,\sigma^*(k),N^{1/k-\ep})\quad (k \ge 6).$$
Weaker conclusions of similar type could be extracted from the earlier work of Boklan \cite{Bok:94}, Ford
\cite{Ford:95}, Heath-Brown \cite{HB:88}, and Wooley \cite{W:92vin}. For a comprehensive discussion of the various ingredients, see \cite[\S 7]{W:FAF}.\par

When $s$ is relatively close to $s_1$, one can obtain conclusions sharper than those available from Theorem \ref{theorem1.4} by including some of the excess variables in a mean value together with an integral over the exceptional set. For example, by adapting the ``slim" technology of Wooley \cite{W:03slim}, along with the refinements of Kawada and Wooley \cite{KW:10rel}, \cite{KW:10dav}, in \S\ref{sect:4} we derive the following refinement for smaller exponents.

\begin{theorem}\label{theorem1.6}
Suppose that $\psi(N)$ grows sufficiently slowly in terms of $\L$ and $\tau$.  Then with the values of $k$, $s$ and $\beta$ from the following table, it follows that for every $\ep > 0$ one has $\tZ_{s,k}(N)\ll N^{\beta+\ep}$.
\end{theorem}

$$\boxed{\begin{matrix} k&3&4&4&4&5&5&5&5&5&5&5\\
s&7&13&14&15&25&26&27&28&29&30&31\\
\beta&\tfrac{1}{3}&\tfrac{5}{8}&\tfrac{1}{2}&\tfrac{7}{16}&\tfrac{3}{4}&\tfrac{7}{10}&\tfrac{13}{20}&\tfrac{3}{5}&\tfrac{23}{40}&\tfrac{11}{20}&\tfrac{3}{8}\end{matrix}}$$
\vskip.2cm
\begin{center}\text{Table 2: Exponents for slim exceptional sets}\end{center}
\vskip.1cm
A direct application of Theorem \ref{theorem1.4} would yield weaker estimates. Thus, for example, whereas Theorem \ref{theorem1.4} shows that $\tZ_{7,3}(N)\ll N^{1/2+\ep}$, one sees from Theorem \ref{theorem1.6} that $\tZ_{7,3}(N) \ll N^{1/3+\ep}$.\par

At the cost of requiring an extra variable, in \S\ref{sect:5} we provide a short-interval analogue of Theorem \ref{theorem1.4}. Here we write $\ctZ_{s,k}^\tau(N,M;\psi;\L)$ for the set of real numbers $\mu\in [N,N+M]$ for which (\ref{1.7}) holds, and $\tZ_{s,k}(N,M)$ for its measure.

\begin{theorem}\label{theorem1.7}
Suppose that $k \ge 3$ and that $(s_1,\sigma_1,U)$ forms an accessible triple for $k$. Then whenever
$$s\ge \max\{k+1,\tfrac{1}{2}s_1\},\quad \sigma<\sigma_1,\quad M\ge N^{1-1/k},$$
and $\psi(N)$ grows sufficiently slowly in terms of $s$, $\sigma$, $k$, $\L$, and $\tau$, one has
$$\tZ_{s+1,k}(N,M) \ll M N^{-(2s-s_1)\sigma/k}U(N^{1/k})^{\ep-1}.$$
\end{theorem}

When $s$ is somewhat smaller than is required to bound the exceptional set $\cZ_{s,k}^\tau (N;\L)$ successfully, we can instead aim for non-trivial lower bounds for the measure of the set of $\mu \in [-N,N]$ for which the inequality (\ref{1.1}) does have a solution. We let $Y_{s,k}(N)$ denote the measure of the set
$$\cY_{s,k}^\tau(N;\L)=[-N,N] \setminus \cZ_{s,k}^\tau(N;\L).$$
As in the analogous questions related to Waring's problem, the lower bounds for $Y_{s,k}(N)$ which we derive in \S\ref{sect:9} depend upon suitable upper bounds for mean values of exponential sums. We say that the exponent $\Delta=\Delta_{s,k}$ is {\it admissible} if for each $\ep>0$, whenever $R \le P^{\eta}$ and $\eta$ is sufficiently small, one has
\begin{equation}\label{1.9}
\int_0^1|g(\alpha)|^{2s}\d\alpha \ll P^{2s-k+\Delta+\ep}.
\end{equation}

\begin{theorem}\label{theorem1.8}
If $\Delta=\Delta_{s,k}$ is admissible, then for every $\ep > 0$ one has
$$Y_{s,k}(N) \gg \tau^2 N^{1-\Delta/k-\ep}.$$
In particular, one has $Y_{3,3}(N) \gg \tau^2 N^{\gamma-\ep}$, where $\gamma = (166-\sqrt{2833})/123 > 11/12$.
\end{theorem}

We note that Theorem \ref{theorem1.8} delivers a non-trivial conclusion even when $\tau$ is an explicit function of $N$, provided only that $\tau^{-1}$ is somewhat smaller than $(N^{1-\Delta/k})^{1/2}$.\par

Finally, we consider the approximation of real numbers by linear combinations of two primes, a topic which may be viewed as an analogue of the binary Goldbach problem. Suppose that $\lambda_1$ and $\lambda_2$ are real numbers with $\lambda_1/\lambda_2$ irrational.  Let $\cZ^*(X;\L,\tau)$ denote the set of real numbers
$\mu \in [0,X]$ for which the inequality
$$|\lambda_1p_1+\lambda_2p_2-\mu| < \tau$$
has no solution in prime numbers $p_1, p_2$. In \S\ref{sect:10} we derive a conclusion free of the spacing condition on $\mu$ present in earlier work of of Br{\"u}dern, Cook and Perelli \cite{BCP:97} and the first author \cite{P:02lfp}.

\begin{theorem}
\label{primes}  One has ${\rm{meas}}(\cZ^*(X;\L,\tau)) = o(X)$.
\end{theorem}

It would appear that a quantitative version of this result is currently inaccessible, although in the special case where $\lambda_1/\lambda_2$ is algebraic, such a conclusion has recently been obtained by Br\"udern, Kawada and Wooley \cite{BKW8}. In this situation, an explicit bound on the failures of the asymptotic formula is achieved by \cite[Theorem 1.5]{BKW8}.\par

The results of this paper are motivated by analogous considerations in Waring's problem, in which for suitable values of $s$ one seeks to represent all sufficiently large integers $n$ in the form $x_1^k+\dots +x_s^k=n$. When $s$ fails to be large enough to establish such a conclusion, one may instead try to bound the number of ``exceptional" integers $n$ that fail to admit such a representation. In particular, if the Hardy-Littlewood method produces representations for all large $n$ whenever $s\ge \cG(k)$, then one can typically show that the number of exceptional integers $n\le X$ is $o(X)$ whenever $s\ge \frac{1}{2}\cG(k)$. There has been considerable recent work aimed at refining our understanding of these exceptional sets. Although quantitative bounds
are ultimately connected to the savings that one can achieve on a suitably defined set of minor arcs, the power and flexibility of the methods has recently been enhanced by technology in which an exponential sum over the set of exceptions is used to better exploit extra variables (see \cite{BKW1}, \cite{BKW3}, \cite{BW:04short}, \cite{KW:10rel}, \cite{KW:10dav}, \cite{Kum:05}, \cite{W:02slimcube}, \cite{W:02slimsquare} and \cite{W:03slim}).\par

Historically, the situation for Diophantine inequalities has been somewhat different. For inequalities of the shape (\ref{1.1}), the method of Davenport and Heilbronn requires one to restrict to a possibly sparse sequence of box sizes defined in terms of the continued fraction convergents to some ratio $\lambda_i/\lambda_j$. Such a
restriction on the box size offers little hope of analysing the set of exceptional $\mu$ for which (\ref{1.1}) has no integer solution. However, work of Freeman \cite{Freem:alb}, \cite{Freem:af}, inspired by the methods of Bentkus and G\"otze \cite{BG:99} and refined by the second author \cite{W:FAF}, has changed the perspective. As a result of these innovations, one is now able to obtain asymptotics for the number of solutions in all sufficiently large boxes, albeit with inexplicit error terms arising from the minor and trivial arcs. Moreover, it transpires that quantitative upper bounds on the measure of the exceptional set are accessible by employing Hardy-Littlewood dissections with respect to various $\lambda_i\alpha$, where $\alpha$ is the variable of integration. This idea has been useful in previous work on inequalities (see for example the proof of \cite[Lemma 2.3]{P:ineq3}, and the amplification procedure of \cite{W:FAF}) when attempting to obtain optimal estimates for mixed mean values. The points for which none of the $\lambda_i \alpha$ satisfy a classical minor arc condition are handled using the Bentkus-G\"otze-Freeman technology, which suffices to show that the
contribution from such $\alpha$ is negligible in comparison to the main term and allows one to exploit the quantitative savings available elsewhere. The method can also be applied to analyse failures of the expected asymptotic formula and to the situation where $\mu$ ranges over a short interval. In contrast to some previous
approaches (see for example \cite{BCP:97}, \cite{P:02lfp}) in which the $\mu$ under consideration were assumed to satisfy an unnatural spacing condition, we make no such assumption and instead integrate
directly over the exceptional set.\par

A word is in order concerning the history of the present paper. An early preprint of this work existed in mid-2007, with exceptional sets containing power savings in a revision prepared following an Oberwolfach meeting a year later. This work had some influence on the paper on thin sequences \cite{BKW8} joint between Br\"udern, Kawada and the second author (see \cite{BKW8add} for subsequent developments). With an eye to the delays created by workload obstructions, we have taken the opportunity in this final version of incorporating the very recent developments pertaining to classical Weyl sums stemming from work of the second author on Vinogradov's mean value theorem (see \cite{W:12war}, \cite{W:12ec} and \cite{W:13ec2}).\par

\section{The kernel of the analysis}\label{sect:2}\label{kernel}
In this section we introduce the elements of the Davenport-Heilbronn method key to our discussion, as well as some basic estimates required during the course of our argument. When $P$ is a large positive number and $1\le Q\le P$ we write
\begin{equation}\label{2.1}
f(\alpha;Q,P)=\sum_{Q<x\le P}e(\alpha x^k).
\end{equation}
Given two suitable positive functions $S(P)$ and $T(P) \le S(P)$ tending to infinity, we define the major arc by
$$\fM=\{\alpha\in\R:|\alpha|\le S(P)P^{-k}\}.$$
We further write
$$\fm=\{\alpha\in\R:S(P)P^{-k}<|\alpha|\le T(P)\}\quad \mbox{and}\quad \ft=\{\alpha\in\R:|\alpha|>T(P)\}$$
for the minor and trivial arcs, and set $L(P)=\max\{1,\log T(P)\}$.  We apply a dissection of this same basic shape for all of our applications, although the specific form of the functions $S$, $T$, and $L$ may change (compare for instance the definitions in \cite[\S3 and \S8]{W:FAF}). For the moment, it suffices to note that this set-up ensures that the major, minor, and trivial arcs give the contributions normally expected in Freeman's version of the Davenport-Heilbronn method when the number of variables is sufficiently large.\par

We next recall the upper and lower bound kernels $K_\pm (\alpha)$ defined by Freeman \cite{Freem:af}. Write $\delta=\tau L(P)^{-1}$ and put
\begin{equation}\label{2.2}
K_\pm (\alpha)=\frac{\sin(\pi \delta \alpha)\sin (\pi (2\tau \pm \delta)\alpha)}{\pi^2\delta \alpha^2}.
\end{equation}
Also, write $U_a(t)$ for the indicator function of the interval $(-a,a)$. Then from \cite[Lemma 1]{Freem:af} and its proof, one finds that
\begin{equation}\label{2.3}
U_{\tau-\delta}(t)\le \int_{-\infty}^\infty e(\alpha t)K_-(\alpha)\d\alpha \le U_\tau (t)
\end{equation}
and
\begin{equation}\label{2.4}
U_{\tau}(t)\le \int_{-\infty}^\infty e(\alpha t)K_+(\alpha)\d\alpha \le U_{\tau+\delta}(t).
\end{equation}
Moreover, one has the bound
\begin{equation}\label{2.5}
K_\pm (\alpha) \ll \min\{\tau,|\alpha|^{-1},\delta^{-1}\alpha^{-2}\}.
\end{equation}

\par In order to facilitate analytic manipulations, it is convenient to work with positive kernels. On recalling (\ref{2.2}), we find that when necessary the kernel $K_\pm(\alpha)$ may be decomposed by means of the relation
\begin{equation}\label{2.6}
|K_\pm(\alpha)|^2=K_1(\alpha)K_2^{\pm}(\alpha),
\end{equation}
where
\begin{equation}\label{2.7}
K_1(\alpha)=\left(\frac{\sin (\pi \delta \alpha)}{\pi \delta \alpha}\right)^{\!\!2}\le \min\{1, (\pi \delta \alpha)^{-2}\}
\end{equation}
and
\begin{equation}\label{2.8}
K_2^{\pm}(\alpha)= \left(\frac{\sin (\pi(2\tau\pm \delta)\alpha)}{\pi \alpha}\right)^{\! \! 2}\le \min\{5\tau^2,(\pi \alpha)^{-2}\}.
\end{equation}
We note that by making a change of variable in \cite[Lemma 14.1]{Bak:DI}, one finds that the Fourier transform of the kernel $K_1$ satisfies
\begin{equation}\label{2.9}
\widehat{K_1}(t)=\int_{-\infty}^\infty e(\alpha t)K_1(\alpha)\d\alpha =\delta^{-1}\max\{0,1-\delta^{-1}|t|\}\le \delta^{-1} U_\delta (t).
\end{equation}

\par Next we introduce exponential integrals with which to encode the sets of real numbers that we seek to test for unrepresentations, or occasionally for representations. When $\cZ$ is any measurable subset of $\R$ and $\eta_\mu$ is any complex-valued function of $\mu$, put
\begin{equation}\label{2.10}
H_{\eta,\cZ}(\alpha)=\int_\cZ \eta_\mu e(-\alpha \mu)\d\mu.
\end{equation}
The following mean value estimates play a critical r\^ole in our arguments.

\begin{lemma}\label{lemma2.1} When $\cZ$ is a set of finite measure $Z$, and $|\eta_{\mu}|=1$ for $\mu \in \cZ$, then
$$\int_{-\infty}^\infty |H_{\eta,\cZ}(\alpha)|^2K_1(\alpha)\d\alpha \le 2Z.$$
\end{lemma}

\begin{proof} In view of (\ref{2.7}) and (\ref{2.10}), it follows from Fubini's Theorem that the mean value in question may be written as
$$I=\int_{\cZ^2} \overline{\eta_{\mu}}{\eta_{\nu}}\int_{-\infty}^{\infty} e(\alpha(\mu-\nu))K_1(\alpha) \d\alpha \d\mu \d\nu = \int_{\cZ^2} \overline{\eta_{\mu}}{\eta_{\nu}} \widehat{K_1}(\mu-\nu) \d\mu \d\nu.$$
Moreover, the upper bound (\ref{2.9}) shows that for every fixed $\nu$ one has $\widehat{K_1}(\mu-\nu)=0$ unless $\nu-\delta < \mu < \nu+\delta$. Consequently,
$$I\le \int_{\cZ^2} \widehat{K_1}(\mu-\nu) \d\mu \d\nu  \le  \delta^{-1} \int_{\cZ}\int_{\nu-\delta}^{\nu+\delta}\d\mu \d\nu = 2Z,$$
and the proof of the lemma is complete.\end{proof}

When $c$ is a positive constant we write $f^{(c)}(\alpha)$ for the exponential sum $f(\alpha;cP,P)$ defined in (\ref{2.1}). The following lemma allows us to handle exceptional sets in short intervals and plays a key r\^ole in the proof of Theorem \ref{theorem1.7}.

\begin{lemma}\label{lemma2.2}
Let $N$ be a large real number, and write $P=N^{1/k}$. In addition, let $c$ and $\lambda$ be non-zero real numbers with $c>0$. Suppose that $M\le |\lambda|(k-1)(cP)^{k-1}$. Then whenever $\cZ$ is a subset of $[N,N+M]$ having measure $Z$, and $|\eta_\mu |=1$ for all $\mu \in \cZ$, one has
$$\int_{-\infty}^\infty |f^{(c)}(\lambda \alpha)H_{\eta,\cZ}(\alpha)|^2K_1(\alpha)\d\alpha \le 2PZ.$$
\end{lemma}

\begin{proof} As in the proof of Lemma \ref{lemma2.1}, the mean value under consideration may be written as
$$I=\int_{\cZ^2}{\overline \eta}_\mu \eta_\nu \int_{-\infty}^\infty|f^{(c)}(\lambda\alpha)|^2e(\alpha(\mu-\nu))K_1(\alpha)\d\alpha \d\mu \d\nu.$$
In view of the relation (\ref{2.9}), one has
\begin{equation}\label{2.11}
I\le \int_{\cZ^2}\delta^{-1}\cW(\mu,\nu)\d\mu\d\nu,
\end{equation}
where $\cW(\mu,\nu)$ denotes the number of solutions of the inequality
$$|\lambda(x^k-y^k)+\mu-\nu|<\delta,$$
with $cP<x,y\le P$. Our hypothesis concerning the size of $M$ ensures that whenever $\mu,\nu \in \cZ$, one has
$$|\mu-\nu|\le M\le |\lambda|(k-1)(cP)^{k-1}.$$
Then if $cP<x,y\le P$ and $x\neq y$, one has
\begin{align*}
|\lambda(x^k-y^k)|&=|\lambda(x-y)(x^{k-1}+x^{k-2}y+\dots +y^{k-1})|>|\lambda|k(cP)^{k-1}\\
&>|\lambda|(k-1)(cP)^{k-1}+\delta\ge |\mu -\nu|+\delta.
\end{align*}
We therefore deduce that $x=y$ for every solution counted by $\cW(\mu,\nu)$, and hence that $\cW(\mu,\nu) \le P$. Moreover, for every fixed $\nu$ one has $\cW(\mu,\nu)=0$ unless $\nu-\delta<\mu<\nu+\delta$. Thus we infer from (\ref{2.11}) that
$$I\le \delta^{-1}P\int_{\cZ}\int_{\nu-\delta}^{\nu+\delta}  \d\mu \d\nu = 2PZ,$$
and the desired conclusion follows.
\end{proof}

We next record a general principle that allows us to extend mean value estimates over subsets of the unit interval to corresponding subsets of the whole real line in the presence of a suitably decaying kernel function.

\begin{lemma}\label{lemma2.3}
Let $X$ and $t$ be non-negative real numbers, and let $\lambda\in \R$. Also, when $\cA \subseteq [1,P] \cap \Z$ and $\fB \subseteq [0,1)$ is measurable, let
$$h(\alpha) =\sum_{x \in \cA}e(\alpha x^k)\quad \mbox{and}\quad \fB_{\lambda}(X)=\{\alpha\in\R:\text{$|\alpha|\ge X$ and $\lambda \alpha \in \fB \mmod{1}$}\}.$$
Then for any real-valued function $K$ satisfying $|K(\alpha)|\ll \min\{1,\alpha^{-2}\}$, one has
$$\int_{\fB_{\lambda}(X)}|h(\lambda \alpha)|^tK(\alpha)\d\alpha \ll_\lambda (1+X)^{-1}\int_\fB |h(\alpha)|^t\d\alpha.$$
\end{lemma}

\begin{proof} After splitting into intervals of the form $\fB_{\lambda}(X) \cap [n,n+1)$ and then making a change of variable, we obtain
$$\int_{\fB_{\lambda}(X)}|h(\lambda \alpha)|^tK(\alpha)\d\alpha \ll |\lambda|^{-1}\sum_{n\ge X}(1+n)^{-2}\int_{\fC_{\lambda}(n)}|h(\alpha)|^t\d\alpha,$$
where $\fC_{\lambda}(n)$ denotes the set of $\alpha \in \fB\mmod{1}$ with $\lambda n\le \alpha< \lambda (n+1)$. Since $\fC_\lambda (n)$ is contained in a union of at most $1+|\lambda|$ translates of $\fB$ of the shape $\fB + j$, the conclusion of the lemma follows from the periodicity of $h(\alpha)$ modulo 1.
\end{proof}

\section{Failures of the asymptotic formula}\label{sect:3}
We illustrate our methods first with the proof of Theorem \ref{theorem1.4}, which concerns the frequency with which the anticipated asymptotic formula (\ref{1.6}) fails for the function $\cN_{s,k}^\tau (P;\L,\mu)$ counting the number of integral solutions of the inequality (\ref{1.1}) with $\x\in [1,P]^s$. Observe that there is no loss of generality in supposing throughout that $\lambda_1/\lambda_2 \not \in \Q$. Suppose that $(s_1,\sigma_1,U)$ is an accessible triple for $k$, and put $S(P)=\min\{U(P),(2k)^{-1}P\}$. We begin by recording a minor arc estimate. Recall the definition (\ref{2.1}) and write $f_i(\alpha)=f(\lambda_i\alpha;0, P)$.

\begin{lemma}\label{lemma3.1} There exists a choice for the function $T(P)$, depending only on $\lambda_1$, $\lambda_2$ and $S(P)$, with the property that
$$\sup_{\alpha \in \fm}|f_1(\alpha)f_2(\alpha)| \ll P^2 T(P)^{-2^{-k-1}}.$$
\end{lemma}

\begin{proof} In view of our definition of the minor arcs $\fm$, the desired conclusion is immediate from \cite[Lemma 2.3]{W:FAF}.
\end{proof}

It follows from the conclusion of this lemma that for all $\alpha\in \fm$, one has
\begin{equation}\label{3.1}
f_j(\alpha)\ll PT(P)^{-6^{-k}}
\end{equation}
for at least one suffix $j\in\{1,2\}$. We define $\fq_j$ to be the set of real numbers $\alpha \in \fm$ for which the upper bound (\ref{3.1}) holds, and then put $\fp_j=\fq_j\cup \ft$, so that $\fm\cup \ft\subseteq \fp_1\cup \fp_2$. In addition, when $1\le i\le s$, we write $\fN_i$ for the set of real numbers $\alpha \in \fm\cup \ft$ for which $\lambda_i\alpha\in \fN$, and $\fn_i=(\fm\cup\ft)\setminus \fN_i$. In order to assist our discussion, we define
$$\upsilon_{ij}=\begin{cases}1,&\text{when $i=j$ and $j\in\{1,2\}$,}\\
0,&\text{otherwise.}\end{cases}$$
Finally, when $t>0$, the set $\fB$ is measurable, and $K$ is integrable, we write
\begin{equation}\label{3.2}
\cM_{i,t}(\fB;K)=\int_\fB |f_i(\alpha)|^tK(\alpha)\d\alpha \quad (1\le i\le s).
\end{equation}

\par We first establish an estimate for the mean value introduced in (\ref{3.2}) of modified major arc type.

\begin{lemma}\label{lemma3.2}
Suppose that $k\ge 3$, $t>k+1$, $1\le i\le s$ and $j\in\{1,2\}$. Then for any fixed $\kappa>0$, one has
$$\cM_{i,t}(\fN_i\cap \fp_j;K_2^\pm)\ll_t P^{t-k}L(P)^{-\kappa s\upsilon_{ij}}$$
and
$$\cM_{i,t}(\fN_i\cap \fp_j;|K_\pm|)\ll_t P^{t-k}L(P)^{1-\kappa s\upsilon_{ij}}.$$
\end{lemma}

\begin{proof} We prove first that when $t>k+1$ and $K(\alpha) \ll \min\{1, \alpha^{-2}\}$, one has
\begin{equation}\label{3.3}
\cM_{i,t}(\fN_i\cap \fp_j;K)\ll_t P^{t-k}L(P)^{-\kappa s\upsilon_{ij}}.
\end{equation}
We may suppose that $t=k+1+2\gamma$, where $\gamma>0$. We put $u=k+1+\gamma$. Then from the proof of Theorem 4.4 and Lemma 4.9 of \cite{V:HL}, on using (4.13) of the latter in place of (4.14), we find that
$$\int_{\fN_i\cap [0,1)}|f_i(\alpha)|^u\d\alpha \ll P^{u-k}.$$
Then by applying (\ref{3.1}) in combination with Lemma \ref{lemma2.3}, we have on the one hand
\begin{align*}
\int_{\fN_i\cap \fq_j}|f_i(\alpha)|^t|K(\alpha)|\d\alpha &\ll \Bigl(\sup_{\alpha\in\fq_j}|f_i(\alpha)|\Bigr)^\gamma \int_{\fN_i\cap [0,1)}|f_i(\alpha)|^u\d\alpha \\
&\ll \left(PT(P)^{-\upsilon_{ij}6^{-k}}\right)^\gamma P^{u-k}\ll P^{t-k}L(P)^{-\kappa s\upsilon_{ij}},
\end{align*}
whilst on the other
\begin{align*}
\int_{\fN_i\cap \ft}|f_i(\alpha)|^t|K(\alpha)|\d\alpha &\ll f(0)^\gamma T(P)^{-1}\int_{\fN_i\cap [0,1)}|f_i(\alpha)|^u\d\alpha \\
&\ll (P^\gamma L(P)^{-\kappa s})P^{u-k}=P^{t-k}L(P)^{-\kappa s}.
\end{align*}
Our claimed bound (\ref{3.3}) follows from (\ref{3.2}) by combining these two upper bounds.\par

The respective conclusions of the lemma follow from (\ref{3.3}) on noting first from (\ref{2.8}) that $K_2^\pm (\alpha)\ll \min\{1,\alpha^{-2}\}$, and second from (\ref{2.5}) that
$$K_\pm(\alpha)\ll \delta^{-1}\min\{1,\alpha^{-2}\}\ll L(P)\min\{1,\alpha^{-2}\}.$$
\end{proof}

We turn next to a corresponding estimate of minor arc flavour.

\begin{lemma}\label{lemma3.3}
Suppose that $(s_1,\sigma_1,U)$ is an accessible triple for $k$, and that $t$ and $\sigma$ are real numbers with $t\ge s_1$ and $\sigma<\sigma_1$. Then for $1\le i\le s$, one has
$$\cM_{i,t}(\fn_i;K_2^\pm)\ll P^{t-k-(t-s_1)\sigma}U(P)^{-1}.$$
\end{lemma}

\begin{proof} On noting (\ref{2.8}), the desired estimate follows by inserting the estimates stemming from (\ref{1.8}) into the conclusion of Lemma \ref{lemma2.3}.\end{proof}

We now embark upon the proof of Theorem \ref{theorem1.4}. Adopt the hypotheses of the statement of the latter theorem, take $N$ to be a positive number sufficiently large in terms of $s$, $k$, $\L$ and $\tau$, and put $P=N^{1/k}$. When $\mu\in (N/2,N]$, we define
$$R_{\pm}(\mu)=\int_{-\infty}^\infty \ftil(\alpha)e(-\alpha \mu)K_\pm(\alpha)\d\alpha ,$$
in which we have written $\ftil(\alpha)=f_1(\alpha)f_2(\alpha)\cdots f_s(\alpha)$. Then it follows from (\ref{2.3}) and (\ref{2.4}) that whenever $\mu\in (N/2,N]$, one has
\begin{equation}\label{3.4}
R_-(\mu)\le \cN_{s,k}^\tau(N^{1/k};\L,\mu)\le R_+(\mu).
\end{equation}
In view of our assumption that $\lambda_i>2$ for some index $i$, the argument leading to \cite[Lemma 6.1]{W:FAF} shows that whenever $\mu \in (N/2,N]$, then
$$\int_{\fM}\ftil(\alpha)e(-\alpha \mu)K_{\pm}(\alpha)\d\alpha = 2\tau\Omega_{s,k}(\L,\mu/N)P^{s-k}+O(P^{s-k}L(P)^{-1}).$$
Here we have made use of the implicit hypothesis that $s\ge k+2$. It follows that whenever $\psi(N)$ grows slowly enough in terms of $s$, $\sigma$, $k$, $\L$ and $\tau$, and in particular, sufficiently slowly in terms of $L(P)$, then
\begin{equation}\label{3.5}
\left|\int_\fM \ftil(\alpha)e(-\alpha \mu)K_\pm(\alpha)\d\alpha-2\tau \Omega_{s,k}(\L,\mu/N)N^{s/k-1}\right|<\tfrac{1}{2}N^{s/k-1}\psi(N)^{-1}.
\end{equation}

Next, write $\cZ=\ctZ_{s,k}^\tau (N;\psi;\L)$, put $Z=\text{meas}(\cZ)$, and consider an element $\mu$ of $\cZ$. Since $\R$ is the disjoint union of $\fM$, $\fm$ and $\ft$, a comparison of (\ref{1.7}) and (\ref{3.5}) leads from (\ref{3.4}) to the conclusion that with $K_\mu=K_{+}$ or $K_\mu=K_{-}$, one has
\begin{equation}\label{3.6}
\left|\int_{\fm\cup \ft}\ftil(\alpha)e(-\alpha \mu)K_\mu (\alpha)\d\alpha\right|\ge \tfrac{1}{2}P^{s-k}\psi(N)^{-1}.
\end{equation}
Using $*$ to denote either $+$ or $-$, denote by $\cZ^*$ the set of $\mu \in \cZ$ for which (\ref{3.6}) holds with $K_{\mu}=K_{*}$, and write $Z^*=\text{meas}(\cZ^*)$. Then it follows that $Z\le Z^++Z^-$, so that for some choice of $*$, either $+$ or $-$, one has $Z\le 2Z^*$. We fix this choice henceforth. For each $\mu \in \cZ^*$, we determine the complex number $\eta_\mu$ by means of the relation
$$\left|\int_{\fm\cup \ft}\ftil(\alpha)e(-\alpha \mu)K_*(\alpha)\d\alpha\right|=\eta_\mu \int_{\fm\cup \ft} \ftil(\alpha)e(-\alpha \mu)K_*(\alpha)\d\alpha.$$
Recall the definition (\ref{2.10}), write $H(\alpha)=H_{\eta,\cZ^*}(\alpha)$, and note that $|\eta_\mu|=1$ for each $\mu\in \cZ^*$. Then by integrating the relation (\ref{3.6}) over the set $\cZ^*$, we find that
$$\int_{\fm \cup \ft}\ftil(\alpha)H(\alpha)K_*(\alpha)\d\alpha\ge \tfrac{1}{2}P^{s-k}\psi(N)^{-1}\int_{\cZ^*}\d\mu.$$
For the sake of convenience, when $\fB$ is measurable, we write
\begin{equation}\label{3.7}
I(\fB)=\int_\fB|\ftil(\alpha)H(\alpha)K_*(\alpha)|\d\alpha.
\end{equation}
In this notation, the last lower bound implies that
\begin{equation}\label{3.8}
I(\fp_1)+I(\fp_2)\ge \tfrac{1}{4}P^{s-k}\psi(N)^{-1}Z.
\end{equation}

\par In the final phase of our proof we obtain estimates for the integrals $I(\fp_j)$, thereby converting the lower bound (\ref{3.8}) into an upper bound for $Z$. Let $j$ be either $1$ or $2$. Then an application of H\"older's inequality leads from (\ref{3.7}) to the bound
\begin{equation}\label{3.9}
I(\fp_j)\le \prod_{i=1}^s\Bigl( \int_{\fp_j}|f_i(\alpha)^sH(\alpha)K_*(\alpha)|\d\alpha\Bigr)^{1/s}.
\end{equation}
Consider an index $i$ with $1\le i\le s$. Since we may suppose that $s\ge k+2$, we deduce from Lemma \ref{lemma3.2} that
\begin{align}
\int_{\fN_i\cap \fp_j}|f_i(\alpha)^sH(\alpha)K_*(\alpha)|\d\alpha &\le H(0)\cM_{i,s}(\fN_i\cap \fp_j;|K_*|)\notag \\
&\ll ZP^{s-k}L(P)^{1-3s\upsilon_{ij}}\label{3.10}.
\end{align}
Next write
\begin{equation}\label{3.11}
\cJ=\int_{-\infty}^\infty |H(\alpha)|^2K_1(\alpha)\d\alpha .
\end{equation}
Then since we may suppose that $s\ge \tfrac{1}{2}s_1$, an application of Schwarz's inequality in combination with Lemmata \ref{lemma2.1} and \ref{lemma3.3} leads via (\ref{2.6}) to the bound
\begin{align}
\int_{\fn_i}|f_i(\alpha)^sH(\alpha)K_*(\alpha)|\d\alpha &\le \cJ^{1/2}\cM_{i,2s}(\fn_i;K_2^*)^{1/2}\notag \\
&\ll Z^{1/2}\left( P^{2s-2k}\Xi_\ep\right)^{1/2},\label{3.12}
\end{align}
where we have written
\begin{equation}\label{3.13}
\Xi_\ep=P^{k-(2s-s_1)(\sigma_1-\ep)}U(P)^{-1}.
\end{equation}

\par On substituting (\ref{3.10}) and (\ref{3.12}) into (\ref{3.9}), we deduce that
$$I(\fp_1)+I(\fp_2)\ll P^{s-k}\left(ZL(P)+Z^{1/2}\Xi_\ep^{1/2}\right)^{1-1/s}\left(ZL(P)^{1-3s}+Z^{1/2}\Xi_\ep^{1/2}\right)^{1/s}.$$
By substituting this bound into the relation (\ref{3.8}), we find that
\begin{align*}
\psi(N)^{-1}Z\ll L(P)^{-2}&Z+L(P)^{1-1/s}\Xi_\ep^{1/(2s)}Z^{(2s-1)/(2s)}\\
&+\Xi_\ep^{(s-1)/(2s)}Z^{(s+1)/(2s)}+\Xi_\ep^{1/2}Z^{1/2}.
\end{align*}
Note that we are at liberty to suppose $\psi(N)$ to be sufficiently small compared to $L(P)$, and further $L(P)$ to be $O(U(P)^\ep)$. Disentangling this inequality, therefore, we conclude from (\ref{3.13}) that for any positive number $\sigma$ with $\sigma<\sigma_1$, one has
\begin{equation}\label{3.14}
Z\ll L(P)^{2s-2}\psi(N)^{2s}\Xi_\ep\ll P^{k-(2s-s_1)\sigma}U(P)^{\ep-1}.
\end{equation}
On recalling that $P=N^{1/k}$, the proof of Theorem \ref{theorem1.4} is complete.
\vskip.3cm

The proof of Theorem \ref{theorem1.5} follows by applying a simplified variant of the above argument, as we now sketch. We make use of the notation applied in the proof of Theorem \ref{theorem1.4} above, fixing $k=3$ and $s=4$. We begin by observing that the argument leading to \cite[Lemma 6.1]{W:FAF} shows on this occasion that whenever $\mu \in (N/2,N]$, then
$$\int_\fM \ftil(\alpha)e(-\alpha \mu)K_\pm (\alpha)\d\alpha =2\tau\Omega_{4,3}(\L,\mu/N)P+O(PL(P)^{-1}).$$
Here we note in particular that when $k=3$ and $s=4$, the inequality (6.4) of \cite{W:FAF} must be replaced by
\begin{align*}
\int_\fM f_1(\alpha)\cdots f_4(\alpha)e(-\alpha \mu)&K_\pm(\alpha)\d\alpha -\int_\fM v_1(\alpha)\cdots v_4(\alpha)e(-\alpha \mu)K_\pm (\alpha)\d\alpha\\
&\ll \sum_{n=1}^\infty n^{-2}\int_{-1/2}^{1/2}P^{7/2}(1+P^3|\alpha|)^{-1}\d\alpha \ll P^{1/2}\log P.
\end{align*}
The latter, achieved using Lemma \ref{lemma2.3}, suffices for the ensuing argument. We therefore see as in (\ref{3.8}) that in the present situation, one has
$$I(\fm\cup\ft)\ge \tfrac{1}{4}P\psi(N)^{-1}Z.$$
The argument of \cite[\S\S4 and 5]{W:FAF}, moreover, shows that
$$\int_{\fm\cup\ft}|f_1(\alpha)\cdots f_4(\alpha)|^2K_2^*(\alpha)\d\alpha \ll P^5L(P)^{-1}.$$
Then on recalling (\ref{3.11}), an application of Schwarz's inequality yields the relation
\begin{align*}
P\psi(N)^{-1}Z\ll I(\fm\cup\ft)&\ll \cJ^{1/2}\Bigl( \int_{\fm\cup\ft}|f_1(\alpha )\cdots f_4(\alpha)|^2K_2^*(\alpha)\d\alpha\Bigr)^{1/2}\\
&\ll P^{5/2}L(P)^{-1/2}Z^{1/2},
\end{align*}
whence $Z\ll P^3\psi(N)^2L(P)^{-1}$. The conclusion of the theorem follows by taking $\psi(N)$ no larger than $L(P)^{1/4}$.

\section{The asymptotic formula for smaller exponents}\label{sect:4}
In order to obtain the sharper results for $3\le k\le 5$ advertised in Theorem \ref{theorem1.6}, we adapt the methods of Wooley \cite{W:03slim} and Kawada and Wooley \cite{KW:10rel}, \cite{KW:10dav}. The following lemma provides the appropriate analogue of \cite[Lemma 6.1]{KW:10rel}, which applies a method of Davenport \cite{Dav:42} to sharpen the conclusion of \cite[Lemma 2.1]{W:03slim} in situations where the exceptional set may be relatively large.

\begin{lemma}\label{lemma4.1} Suppose that $k\ge 3$, that $1\le j\le k-2$, and that $\lambda$ is a non-zero real number. Let $\cZ$ be a subset of $[-P^k, P^k]$ with ${\rm meas}(\cZ)=Z$. Then for every $\ep>0$, one has
$$\int_{-\infty}^\infty |f(\lambda \alpha)^{2^j}H_{\eta,\cZ}(\alpha)^2|K_1(\alpha)\d\alpha \ll P^{2^j}(P^{-1}Z+P^{-1-j/2+\ep}Z^{3/2}).$$
\end{lemma}

\begin{proof} We apply Weyl differencing as in the proof of \cite[Lemma 2.1]{W:03slim} to deduce that, for suitably defined intervals $I_j(\h)\subseteq [1,P]\cap \Z$, one has
$$|f(\lambda \alpha)|^{2^j}\le (2P)^{2^j-j-1}\sum_{\h \in (-P,P)^j}\sum_{x \in I_j(\h)}e(\lambda \alpha h_1\cdots h_j p_j(x;\h)),$$
where $p_j$ is a polynomial of degree $k-j$ in $x$. As in the proof of Lemma \ref{lemma2.2}, it therefore follows that
\begin{equation}\label{4.1}
\int_{-\infty}^{\infty}|f(\lambda \alpha)^{2^j}H_{\eta,\cZ}(\alpha)^2|K_1(\alpha)\d\alpha \ll P^{2^j-j-1}\int_{\cZ^2}\delta^{-1}\cW(\mu,\nu)\d\mu \d\nu,
\end{equation}
where $\cW(\mu,\nu)$ denotes the number of integral solutions of the inequality
$$|\lambda h_1\cdots h_jp_j(x;\h)+\mu-\nu|<\delta,$$
with $|h_i|<P$ $(1\le i\le j)$ and $1\le x\le P$. We let $\cW_0(\mu,\nu)$ denote the number of solutions $x$, $\h$ counted by $\cW(\mu,\nu)$ with $h_1\cdots h_j=0$, and write $\cW_1(\mu,\nu)$ for the corresponding number of solutions with $h_1\cdots h_j\ne 0$.\par

The analysis of $\cW_0(\mu,\nu)$ is straightforward. There are at most $O(P^j)$ choices for $x$ and $\h$ with $h_1\cdots h_j=0$, and so $\cW_0(\mu,\nu)\ll P^jU_{\delta}(\mu-\nu)$. Then it follows as in the proof of Lemma \ref{lemma2.2} that
\begin{equation}\label{4.2}
\int_{\cZ^2}\delta^{-1}\cW_0(\mu,\nu)\d\mu \d\nu \ll P^jZ.
\end{equation}

\par Turning our attention next to $\cW_1(\mu,\nu)$, we denote by $\rho(\beta,\h)$ the number of solutions of the inequality
$$|\lambda h_1\cdots h_jp_j(x;\h)+\beta|<\delta,$$
with $1\le x\le P$. Then one has
$$\cW_1(\mu,\nu)=\sideset{}{'}\sum_\h \rho(\mu-\nu,\h),$$
where we write $\sum'$ to denote the sum over integral $j$-tuples $\h$ with $0<|h_i|<P$ $(1 \le i \le j)$. We also find it convenient to write
$$\cI_1=\int_{\cZ^2}\cW_1(\mu,\nu)\d\mu \d\nu.$$
On applying Schwarz's inequality, followed by Cauchy's inequality, we deduce that
\begin{equation}\label{4.3}
\cI_1\le \biggl(\int_\cZ \d\nu \biggr)^{1/2}\biggl(\int_\cZ \biggl|\sideset{}{'}\sum_\h \int_\cZ \rho(\mu-\nu,\h)\d\mu \biggr|^2\d\nu \biggr)^{1/2}\le Z^{1/2}P^{j/2}\cI_2^{1/2},
\end{equation}
where
$$\cI_2=\int_\cZ\int_{\cZ^2}\sideset{}{'}\sum_\h \rho(\mu_1-\nu,\h)\rho(\mu_2-\nu,\h)\d\mu_1\d\mu_2\d\nu.$$

\par Write $\cV(\mu_1,\mu_2,\nu)$ for the number of solutions of the simultaneous inequalities
$$|\lambda h_1\cdots h_jp_j(x_1;\h)+\mu_1-\nu|<\delta \quad \mbox{and} \quad |\lambda h_1\cdots h_jp_j(x_2;\h)+\mu_2-\nu|<\delta,$$
with $0<|h_i|<P$ $(1\le i\le j)$ and $1\le x_1,x_2\le P$. In addition, denote by $\cV_0(\mu_1,\mu_2,\nu)$ the number of solutions $\x,\h$ counted by $\cV(\mu_1,\mu_2,\nu)$ in which $p_j(x_1;\h)=p_j(x_2;\h)$, and by $\cV_1(\mu_1,\mu_2,\nu)$ the corresponding number of solutions with $p_j(x_1;\h)\ne p_j(x_2;\h)$. It follows that
\begin{equation}\label{4.4}
\cI_2=\int_\cZ \int_{\cZ^2}\biggl(\cV_0(\mu_1,\mu_2,\nu)+\cV_1(\mu_1,\mu_2,\nu)\biggr)\d\mu_2\d\mu_1\d\nu.
\end{equation}

\par Plainly, for each fixed choice of $\h$, $x_2$, and $\nu$, one has
\begin{equation}\label{4.5}
\int_\cZ U_\delta(\lambda h_1\cdots h_jp_j(x_2;\h)+\mu_2-\nu)\d\mu_2\le 2\delta.
\end{equation}
For a given integer $x_1$, there are $O(1)$ values of $x_2$ satisfying $p_j(x_1;\h)=p_j(x_2;\h)$, and hence by summing over $\h$ and $\x$ we deduce from (\ref{4.5}) that
\begin{equation}\label{4.6}
\int_\cZ \int_{\cZ^2}\cV_0(\mu_1,\mu_2,\nu)\d\mu_2\d\mu_1\d\nu \ll \delta \int_\cZ \int_\cZ \cW_1(\mu_1,\nu)\d\mu_1\d\nu =\delta \cI_1.
\end{equation}

\par For a solution $\h$, $\x$ counted by $\cV_1(\mu_1,\mu_2,\nu)$, on the other hand, we observe that the quantity $n=h_1\cdots h_j(p_j(x_1;\h)-p_j(x_2;\h))$ is a nonzero integer satisfying the inequality
\begin{equation}\label{4.7}
|\lambda n+\mu_1-\mu_2|<2\delta.
\end{equation}
Hence for each fixed non-zero value of $n$, a divisor function estimate shows that there are $O(P^\ep )$ possibilities for $\h$ and $\x$. Moreover, for a given $\mu_1$ and $\mu_2$, there are at most $1+4\delta|\lambda|^{-1}$ integers $n$ for which (\ref{4.7}) holds.
After applying the obvious analogue of (\ref{4.5}) to integrate over $\nu$, we therefore deduce that
\begin{align}\label{4.8}
\int_\cZ \int_{\cZ^2}\cV_1(\mu_1,\mu_2,\nu)\d\mu_2\d\mu_1\d\nu &\ll \delta P^\ep \int_{\cZ^2}\sum_{n\in\Z}U_{2\delta}(\lambda n+\mu_1-\mu_2)\d\mu_1\d\mu_2 \notag \\ &\ll \delta P^\ep Z^2.
\end{align}
Consequently, by substituting this estimate together with (\ref{4.6}) into (\ref{4.4}), we obtain the bound
\begin{equation}\label{4.9}
\cI_2\ll \delta(\cI_1+P^\ep Z^2).
\end{equation}

\par Finally, by employing (\ref{4.9}) within (\ref{4.3}), we see that
$$\cI_1\ll \bigl(\delta P^jZ(\cI_1+P^\ep Z^2)\bigr)^{1/2},$$
whence
$$\cI_1\ll \delta P^jZ+\delta^{1/2}P^{j/2+\ep}Z^{3/2}.$$
The lemma now follows from (\ref{4.1}) and (\ref{4.2}), on noting that
$$\delta^{-1/2}\ll L(P)^{1/2} \ll P^\ep.$$
\end{proof}

Our result for quintic forms in $31$ variables instead makes use of the following analogue of \cite[Lemma 3.1]{W:03slim}.

\begin{lemma}\label{lemma4.2}
Suppose that $k\ge 3$, that $2\le j \le k-1$, and that $\lambda$ and $\gamma$ are non-zero real numbers. Let $\cZ$ be a subset of $[-P^k, P^k]$ with ${\rm meas}(\cZ)=Z$. Then for every $\ep>0$, one has
$$\int_{-\infty}^\infty |f(\lambda \alpha)^{2^j}f(\gamma \alpha)^{2^{j-1}}H_{\eta, \cZ}(\alpha)^2|K_1(\alpha)\d\alpha \ll P^{2^j+2^{j-1}}(P^{-2}Z + P^{-1-j+\ep}Z^2).$$
\end{lemma}

\begin{proof} By proceeding as in the proof of Lemma \ref{lemma4.1}, we find that the integral $I$ under consideration satisfies
\begin{equation}\label{4.10}
I\ll P^{2^j-j-1}\int_{\cZ^2}\delta^{-1}\cW(\mu, \nu)\d\mu \d\nu,
\end{equation}
where $\cW(\mu,\nu)$ denotes the number of solutions of the inequality
$$\biggl|\lambda h_1 \cdots h_jp_j(z;\h)+\gamma \sum_{i=1}^{2^{j-2}}(x_i^k-y_i^k)+\mu-\nu\biggr|<\delta,$$
with $|h_l| < P$ $(1\le l\le j)$, $1\le x_i, y_i\le P$ $(1\le i\le 2^{j-2})$, and $1\le z\le P$. Here, the polynomial $p_j(z;\h)$ that arises from the Weyl differencing process has degree $k-j$ in $z$. We let $\cW_0(\mu,\nu)$ denote the number of solutions $\h$, $\x$, $\y$, $z$ counted by $\cW(\mu,\nu)$ with $h_1\cdots h_jp_j(z;\h)=0$, and we write $\cW_1(\mu,\nu)$ for the corresponding number of solutions with $h_1\cdots h_jp_j(z;\h)\ne 0$.\par

Consider first a solution $\h$, $\x$, $\y$, $z$ counted by $\cW_0(\mu,\nu)$. There are $O(P^j)$ choices for $\h$ and $z$ satisfying $h_1\cdots h_jp_j(z;\h)=0$, and we therefore see that
\begin{equation}\label{4.11}
\cW_0(\mu,\nu)\ll P^j\cV(\mu,\nu),
\end{equation}
where $\cV(\mu,\nu)$ denotes the number of solutions of the inequality
$$\biggl|\gamma\sum_{i=1}^{2^{j-2}}(x_i^k-y_i^k)+\mu-\nu\biggr|<\delta,$$
with $1\le x_i,y_i\le P$ $(1\le i\le 2^{j-2})$. It follows from (\ref{2.4}) that
$$\cV(\mu,\nu)\ll \int_{-\infty}^{\infty}|f(\gamma \alpha)|^{2^{j-1}}e(\alpha(\mu-\nu))K_{+}(\alpha)\d\alpha,$$
where we take $\tau =\delta$ in the definition of $K_{+}$. On substituting this estimate into (\ref{4.11}), we conclude thus far that
\begin{equation}\label{4.12}
\int_{\cZ^2}\cW_0(\mu,\nu)\d\mu\d\nu\ll P^j\int_{-\infty}^{\infty}|f(\gamma \alpha)^{2^{j-1}}H_{\eta,\cZ}(\alpha)^2|K_{+}(\alpha)\d\alpha.\end{equation}

\par Next, by applying Weyl differencing together with the second inequality of (\ref{2.4}), we find that
\begin{equation}\label{4.13}
\int_{-\infty}^{\infty}|f(\gamma \alpha)^{2^{j-1}}H_{\eta,\cZ}(\alpha)^2|K_{+}(\alpha)\d\alpha\ll P^{2^{j-1}-j}\int_{\cZ^2}\cX(\mu,\nu)\d\mu\d\nu ,
\end{equation}
where $\cX(\mu,\nu)$ denotes the number of solutions of the inequality
$$|\gamma g_1\cdots g_{j-1}p_{j-1}(w;\g)+\mu-\nu|<2\delta,$$
with $|g_i|<P$ $(1\le i\le j-1)$ and $1\le w\le P$. Here the polynomial $p_{j-1}(w;\g)$ produced by the differencing operation is a polynomial of degree $k-j+1$ in $w$. Then by proceeding as in the arguments leading to (\ref{4.2}) and (\ref{4.8}) in the proof of Lemma \ref{lemma4.1}, one finds that
$$\int_{\cZ^2}\cX(\mu, \nu)\d\mu \d\nu \ll \delta P^{j-1}Z+P^\ep Z^2.$$
It therefore follows from (\ref{4.12}) and (\ref{4.13}) that
\begin{equation}\label{4.14}
\int_{\cZ^2}\cW_0(\mu,\nu)\d\mu \d\nu \ll P^{2^{j-1}}(\delta P^{j-1}Z+P^{\ep}Z^2).
\end{equation}

\par Consider next a solution $\h$, $\x$, $\y$, $z$ counted by $\cW_1(\mu,\nu)$. Given any fixed one amongst the $O(P^{2^{j-1}})$ possible choices for $\x$ and $\y$, write
$$\beta(\x,\y)=\gamma \sum_{i=1}^{2^{j-2}}(x_i^k-y_i^k).$$
The number of available choices for $\h$ and $z$ is then equal to the number of solutions of the equation
$$h_1\cdots h_jp_j(z;\h)=n,$$
with $|h_i|<P$ $(1\le i\le j)$, $1\le z\le P$, and $n$ a non-zero integer satisfying the inequality
$$|\lambda n+\beta(\x,\y)+\mu-\nu|<\delta.$$
It follows that a divisor function estimate once again yields the upper bound
\begin{align}\label{4.15}
\int_{\cZ^2}\cW_1(\mu, \nu)\d\mu \d\nu &\ll P^{2^{j-1}+\ep}\int_{\cZ^2}\sum_{n\in \Z}U_\delta(\lambda n+\beta(\x,\y)+\mu-\nu)\d\mu \d\nu \notag \\
&\ll P^{2^{j-1}+\ep}Z^2.
\end{align}

\par It only remains now to combine (\ref{4.14}) and (\ref{4.15}) within (\ref{4.10}), and we obtain the bound
$$I\ll P^{2^j-j-1}\bigl(P^{2^{j-1}}(P^{j-1}Z+\delta^{-1}P^\ep Z^2)\bigr).$$
The lemma now follows on recalling again that $\delta^{-1}\ll L(P)\ll P^\ep$.\end{proof}

The conclusion of Theorem \ref{theorem1.6} follows from the following more general result.

\begin{theorem}\label{theorem4.3}
Suppose that $k\ge 3$ and that $(s_1,\sigma_1,U)$ forms an accessible triple. Then whenever $s\ge \frac{1}{2}s_1+2^{k-3}$, $\sigma<\sigma_1$, and $\psi(N)$ grows sufficiently slowly in terms of $s$, $\sigma$, $k$, $\L$, and $\tau$, one has
$$\tZ_{s,k}(N)\ll N^{1-1/k-(2s-s_1-2^{k-2})\sigma/k}U(N^{1/k})^{\ep-1}+N^{1-2(2s-s_1-2^{k-2})\sigma/k+\ep}U(N^{1/k})^{-2}.$$
Moreoever, when instead $k\ge 4$ and $(2s-s_1-\frac{3}{8}2^k)\sigma_1>1$, one has
$$\tZ_{s,k}(N)\ll N^{1-2/k-(2s-s_1-\frac{3}{8}2^k)\sigma/k}U(N^{1/k})^{-1}.$$
\end{theorem}

\begin{proof} We follow the argument of the proof of Theorem \ref{theorem1.4} from \S3, economising on details for the sake of concision. Recalling the definitions of $\fp_j$, and of $I(\fB)$ from (\ref{3.7}), we find that $Z$, the measure of ${\widetilde \cZ}_{s,k}^\tau(N;\psi;\L)$, again satisfies the relation (\ref{3.8}). Let $j$ be either $1$ or $2$. Then an application of H\"older's inequality again yields the bound (\ref{3.9}). Consider an index $i$ with $1\le i\le s$. Since we may suppose that $s>k+1$, we again obtain (\ref{3.10}) as a consequence of Lemma \ref{lemma3.2}. Next write
$$\cJ_i=\int_{-\infty}^\infty |f_i(\alpha)|^{2^{k-2}}|H(\alpha)|^2K_1(\alpha)\d\alpha.$$
Since we may also suppose that $s-2^{k-3}\ge \tfrac{1}{2}s_1$, an application of Schwarz's inequality in combination with Lemmata \ref{4.1} and \ref{lemma3.3} leads to the bound
\begin{align*}
\int_{\fn_i}|f_i(\alpha)^sH(\alpha)K_*(\alpha)|\d\alpha &\le \cJ_i^{1/2}\cM_{i,2s-2^{k-2}}(\fn_i;K_2^*)^{1/2}\\
&\ll \left( P^{2^{k-2}}(P^{-1}Z+P^{\ep-k/2}Z^{3/2})\right)^{1/2}\left( P^{2s-2^{k-2}-2k}\Xi_\ep\right)^{1/2},
\end{align*}
where we have written
$$\Xi_\ep=P^{k-(2s-s_1-2^{k-2})(\sigma_1-\ep)}U(P)^{-1}.$$
Thus we obtain the bound
\begin{equation}\label{4.16}
\int_{\fn_i}|f_i(\alpha)^sH(\alpha)K_*(\alpha)|\d\alpha \ll P^{s-k}(P^{-1}Z+P^{\ep-k/2}Z^{3/2})^{1/2}\Xi_\ep^{1/2}.
\end{equation}

\par On substituting (\ref{3.10}) and (\ref{4.16}) into (\ref{3.9}), and thence into (\ref{3.8}), we deduce that
\begin{align*}
P^{s-k}\psi(N)^{-1}Z\ll P^{s-k}&\left(ZL(P)+\left(P^{-1/2}Z^{1/2}+P^{\ep-k/4}Z^{3/4}\right)\Xi_\ep^{1/2}\right)^{1-1/s}\\
&\,\times \left(ZL(P)^{1-3s}+\left(P^{-1/2}Z^{1/2}+P^{\ep-k/4}Z^{3/4}\right)\Xi_\ep^{1/2}\right)^{1/s}.
\end{align*}
This inequality may be disentangled to show that
$$Z\ll L(P)^{s-1}\psi(N)^s\left(P^{-1/2}Z^{1/2}+P^{\ep-k/4}Z^{3/4}\right)\Xi_\ep^{1/2}.$$
We may suppose that $\psi(N)$ is sufficiently small compared to $L(P)$, and that $L(P)\ll U(P)^\ep$. Further disentangling therefore shows that for any positive number $\sigma$ with $\sigma<\sigma_1$, one has
\begin{align*}
Z&\ll P^{-1}L(P)^{2s-2}\psi(N)^{2s}\Xi_\ep+P^{\ep-k}L(P)^{4s-4}\psi(N)^{4s}\Xi_\ep^2\\
&\ll P^{k-1-(2s-s_1-2^{k-2})\sigma}U(P)^{\ep-1}+P^{k-2(2s-s_1-2^{k-2})\sigma+\ep}U(P)^{-2}.
\end{align*}
On recalling that $P=N^{1/k}$, the proof of the first estimate of Theorem \ref{theorem4.3} is complete.\par

The second estimate of the theorem follows in like manner, once one substitutes $\frac{3}{16}2^k$ for $2^{k-3}$, and Lemma \ref{lemma4.2} with $j=k-2$ for Lemma \ref{lemma4.1}, throughout. Thus one obtains
\begin{align*}
P^{s-k}\psi(N)^{-1}Z\ll P^{s-k}&\left(ZL(P)+\left(P^{-1}Z^{1/2}+P^{\ep-(k-1)/2}Z\right)\Xi_\ep^{1/2}\right)^{1-1/s}\\
&\,\times \left(ZL(P)^{1-3s}+\left(P^{-1}Z^{1/2}+P^{\ep-(k-1)/2}Z\right)\Xi_\ep^{1/2}\right)^{1/s},
\end{align*}
where now
$$\Xi_\ep=P^{k-(2s-s_1-\frac{3}{8}2^k)(\sigma_1-\ep)}U(P)^{-1}.$$
We therefore conclude that whenever $(2s-s_1-\tfrac{3}{8}2^k)\sigma_1>1$ and $\sigma$ is a positive number with $\sigma<\sigma_1$, then
$$Z\ll P^{k-2-(2s-s_1-\frac{3}{8}2^k)\sigma}U(P)^{-1}.$$
The proof of the second estimate of the theorem therefore follows again from the relation $P=N^{1/k}$.
\end{proof}

The conclusion of Theorem \ref{theorem1.6} follows directly from Theorem \ref{theorem4.3} on making use of the accessible triples recorded in the preamble to the statement of the former theorem. For all but the last column in the table, one makes use of the first estimate supplied by Theorem \ref{theorem4.3}, and for the last column one applies the second estimate.

\section{The asymptotic formula in short intervals}\label{sect:5}
The key idea of the previous section, in which some of the excess variables were allocated to mean values involving $|H_{\eta,\cZ}(\alpha)|^2$, can also be used to produce short interval results. In this section, we establish Theorem \ref{theorem1.7} by adjusting the analysis of Section \ref{sect:3} to allow for an application of Lemma \ref{lemma2.2} in place of Lemma \ref{lemma2.1}.\par

Unless indicated otherwise, we adopt the notation of \S3. Recall the hypotheses of the statement of Theorem \ref{theorem1.7}, and suppose in particular that $(s_1,\sigma_1,U)$ forms an accessible triple for $k$. Let $\cZ=\ctZ_{s+1,k}^\tau(N,M;\psi;\L)$, and write
$$\lambda=\max_{1\le i\le s+1}|\lambda_i|\quad \text{and}\quad c=(2(s+1)\lambda)^{-1}.$$
We begin by observing that in our proof of Theorem \ref{theorem1.7}, it clearly suffices to show that whenever $M\le \lambda (k-1)(cP)^{k-1}$, one has
$${\rm{meas}}(\cZ)\ll MP^{-(2s-s_1)(\sigma_1-\ep)}U(P)^{\ep-1}.$$
In order to confirm this statement, we observe that if $M>\lambda(k-1)(cP)^{k-1}$, then we may divide the interval $[N,N+M]$ into $O(MP^{1-k})$ intervals of length at most $\lambda (k-1)(cP)^{k-1}$ on which the former bound can be applied. Summing the contributions from all such intervals, the desired bound for the exceptional set follows on noting the trivial upper bound $\lambda(k-1)(cP)^{k-1}<N^{1-1/k}$ that follows on recalling that $P=N^{1/k}$.\par

We now launch the proof of Theorem \ref{theorem1.7} in earnest. Suppose that $\mu$ lies in the interval $[P^k,P^k+M]$ and that the integers $x_1,\ldots ,x_{s+1}$ satisfy
$$|\lambda_1x_1^k+\dots +\lambda_{s+1}x_{s+1}^k-\mu|<\tau +\delta.$$
Then one has
$$\max_{1\le i\le s+1}|x_i|>(2(s+1)\lambda)^{-1/k}P\ge cP.$$
Write $\fg_i(\alpha)$ for $f(\lambda_i\alpha;cP,P)$ and $\fh_i(\alpha)$ for $f(\lambda_i\alpha;0,P)$, in the notation of (\ref{2.1}). Then we deduce from (\ref{2.3}) and (\ref{2.4}) that
\begin{equation}\label{5.1}
\int_{-\infty}^\infty \prod_{i=1}^{s+1}(\fh_i(\alpha)-\fg_i(\alpha))e(-\alpha\mu)K_{\pm}(\alpha)\d\alpha=0.
\end{equation}
Next, on writing
$$\tfh(\alpha)=\prod_{i=1}^{s+1}\fh_i(\alpha),$$
one finds that (\ref{2.3}) and (\ref{2.4}) yield the inequalities
$$\int_{-\infty}^\infty \tfh(\alpha)e(-\alpha\mu)K_-(\alpha)\d\alpha \le \cN_{s+1,k}^\tau (P;\L,\mu)\le \int_{-\infty}^\infty \tfh(\alpha)e(-\alpha\mu)K_+(\alpha)\d\alpha.$$
It therefore follows by subtracting (\ref{5.1}) that
$$\sum_{\fJ \neq \emptyset}(-1)^{|\fJ|+1}\cR_\fJ^-(\mu;\R)\le \cN_{s+1,k}^\tau(P;\L,\mu)\le \sum_{\fJ \neq \emptyset}(-1)^{|\fJ|+1}\cR_\fJ^+(\mu;\R),$$
where the summations are over subsets $\fJ$ of $\{1, \dots, s+1\}$, and where
$$\cR_\fJ^\pm (\mu;\fB)=\int_\fB \prod_{i\in \fJ}\fg_i(\alpha)\prod_{j\not \in \fJ}\fh_j(\alpha)e(-\alpha\mu)K_\pm(\alpha)\d\alpha.$$

\par Applying the analysis leading to \cite[Lemma 6.1]{W:FAF}, as in the argument following Br\"udern and Wooley \cite[equation (6.6)]{BW:04short}, we find that whenever $s\ge k$, one has
$$\sum_{\fJ \neq \emptyset}(-1)^{|\fJ|+1}\cR_\fJ^\pm (\mu;\fM)=2\tau\Omega_{s,k}(\L,\mu/N)P^{s+1-k}+O(P^{s+1-k}L(P)^{-1}).$$
Thus if $\mu \in \cZ$, then with $\cR_{\fJ}^\mu$ denoting either $\cR_\fJ^+$ or $\cR_\fJ^-$, one has the relation
$$\sum_{\fJ \neq \emptyset}|\cR_\fJ^\mu(\mu;\fm \cup \ft)|>\tfrac{1}{2}P^{s+1-k}\psi(N)^{-1},$$
whence for some $\fJ\ne \emptyset$ one has
\begin{equation}\label{5.2}
|\cR_\fJ^\mu(\mu;\fm \cup \ft)|>2^{-s-2}P^{s+1-k}\psi(N)^{-1}.
\end{equation}
As in the proof of Theorem \ref{theorem1.4} from \S3, with $*$ equal to $+$ or $-$, we denote by $\cZ_\fJ^*$ the set of $\mu \in \cZ$ for which (\ref{5.2}) holds with $\cR_\fJ^\mu=\cR_\fJ^*$, and we write $Z_\fJ^*=\text{meas}(\cZ_\fJ^*)$. It follows that for some choices of $*$ and $\fJ$, one has $Z\le 2^{s+2}Z_\fJ^*$. We fix these choices of $*$ and $\fJ$ henceforth. For each $\mu \in \cZ_\fJ^*$, we then determine the complex number $\eta_\mu$ of modulus $1$ by means of the relation
$$|\cR_\fJ^*(\mu;\fm \cup \ft)|=\eta_\mu \cR_\fJ^*(\mu;\fm \cup \ft).$$
Integrating (\ref{5.2}) over $\cZ^*_\fJ$ gives the upper bound
\begin{equation}\label{afshortsetup}
P^{s+1-k}\psi(N)^{-1}Z\ll \int_{\fm \cup \ft}\left|\prod_{i \in \fJ}\fg_i(\alpha)\prod_{j \not \in \fJ}\fh_j(\alpha)H(\alpha)K_*(\alpha)\right|\d\alpha,
\end{equation}
where in the notation introduced in (\ref{2.10}), we have written $H(\alpha)=H_{\eta,\cZ_\fJ^*}(\alpha)$.\par

Suppose first that there exists an index $i\in \fJ$ with $i\not\in \{1,2\}$. In this situation, by relabelling indices if necessary, there is no loss of generality in supposing that $s+1\in \fJ$. We economise on exposition by writing
$$f_i(\alpha)=\begin{cases}\fg_i(\alpha),&\text{when $i\in \fJ$,}\\
\fh_i(\alpha),&\text{when $i\not\in \fJ$.}\end{cases}$$
We then put
$$\tf(\alpha)=f_1(\alpha)\cdots f_s(\alpha)\quad \text{and}\quad H^\dagger(\alpha)=f_{s+1}(\alpha)H(\alpha),$$
and define
\begin{equation}\label{5.4}
I(\fB)=\int_\fB |\tf(\alpha)H^\dagger(\alpha)K_*(\alpha)|\d\alpha.
\end{equation}
In this way we find that (\ref{afshortsetup}) may be rewritten in the shape
$$I(\fp_1)+I(\fp_2)\gg P^{s-k}\psi(N)^{-1}Z^\dagger,$$
where $Z^\dagger=PZ$. Next define
$$\cJ^\dagger =\int_{-\infty}^\infty |H^\dagger(\alpha)|^2K_1(\alpha)\d\alpha,$$
and note that Lemma \ref{lemma2.2} supplies the estimate $\cJ^\dagger\le 2PZ=2Z^\dagger$. Recognising that the argument of the proof of Theorem \ref{theorem1.4} leading from (\ref{3.8}) to (\ref{3.14}) may be repeated essentially verbatim, save that $H(\alpha)$, $\cJ$ and $Z$ are to be decorated throughout by obelisks, the estimation of $I(\fp_j)$ may be completed without incident for $j=1,2$. In this way, we may conclude that whenever $\psi(N)$ grows sufficiently slowly and $\sigma$ is any positive number with $\sigma<\sigma_1$, then
$$PZ=Z^\dagger \ll P^{k-(2s-s_1)\sigma }U(P)^{\ep-1},$$
and the proof of Theorem \ref{theorem1.7} is completed by recalling that $P=N^{1/k}$.\par

It remains to deal with the situation where $\fJ\subseteq\{1,2\}$. In this case, there is no loss of generality in supposing that $1\in \fJ$. We then put
$$\tf(\alpha)=f_2(\alpha)\cdots f_{s+1}(\alpha)\quad \text{and}\quad H^\dagger(\alpha)=f_1(\alpha)H(\alpha),$$
and define $I(\fB)$ again as in (\ref{5.4}). In the current situation, one has
$$\sup_{\alpha\in\fp_1}|H^\dagger(\alpha)|\le Z\sup_{\alpha \in \fp_1}|f_1(\alpha)|\ll Z^\dagger L(P)^{-3s},$$
and hence for $1\le i\le s+1$ we obtain the estimate
\begin{align*}
\int_{\fN_i\cap\fp_1}|f_i(\alpha)^sH^\dagger(\alpha)K_*(\alpha)|\d\alpha &\ll Z^\dagger L(P)^{-3s}\cM_{i,s}(\fN_i\cap\fp_1;|K_*|)\\
&\ll Z^\dagger P^{s-k}L(P)^{1-3s}
\end{align*}
as a substitute for (\ref{3.10}). In all other respects, the argument outlined in the previous paragraph remains valid following a transparent relabelling of variables, and thus the conclusion of Theorem \ref{theorem1.7} follows even in this case.\par

\section{Exceptions to solubility}\label{sect:6}
If one is interested only in sets on which a solution to (\ref{1.1}) fails to exist, rather than sets on which the asymptotic formula fails, techniques involving diminishing ranges and smooth numbers offer additional flexibility in the analysis. Diminishing ranges allow for higher-moment generalizations of Lemma \ref{lemma2.2}, and the resulting diagonal behavior permits shorter intervals than those discussed in Theorem \ref{theorem1.7}. Furthermore, the use of smooth numbers reduces the number of variables required to obtain best-possible mean-value estimates for the corresponding exponential sums over $k$th powers.\par

We let $N$ be a large positive number, set $P=N^{1/k}$, and put
$$\lambda=\max_{1\le i\le s+t}|\lambda_i|.$$
We are again free to suppose that $\lambda_1/\lambda_2\not\in \Q$. We take $R=P^\eta$ with $\eta>0$ sufficiently small in terms of the ambient parameters. Suppose that $(s_0,\sigma_0)$ forms a smooth accessible pair for $k$, and put $S(P)=(\log P)^\nu$, wherein $\nu$ is a sufficiently small positive number. Define $g(\alpha)$ as in (\ref{1.3}) and $f(\alpha;Q,P)$ by means of (\ref{2.1}). Further, let $c$ be a fixed positive number sufficiently large in terms of $t$, $k$ and $\L$, and write
$$P_j=c^{-j}P^{(1-1/k)^{j-1}}, \quad g_j(\alpha)=g(\lambda_j\alpha)\quad \mbox{and}
\quad \ff_j(\alpha)=f(\lambda_{s+j}\alpha;P_j,2P_j) .$$
Let $\cZ=\cZ_{s+t,k}^\tau (N,M;\L)$ denote the set of real numbers $\mu\in [N,N+M]$ for which the inequality
$$|\lambda_1x_1^k+\cdots +\lambda_{s+t}x_{s+t}^k-\mu|<\tau$$
has no integer
solution, and write $H_\cZ(\alpha)$ for $H_{\eta,\cZ}(\alpha)$ when the function $\eta$ is identically 1. The following lemma provides a natural extension of Lemma \ref{lemma2.2}.

\begin{lemma}\label{lemma6.1}
Let $t$ be a positive integer, and suppose that $M\le P_t^{k-1}$. In addition, let $\cZ$ be a subset of $[N,N+M]$ of measure $Z$. Then
$$\int_{-\infty}^\infty |\ff_1(\alpha)\cdots \ff_t(\alpha)H_\cZ (\alpha)|^2K_1(\alpha)\d\alpha \le 2P_1\cdots P_tZ.$$
\end{lemma}

\begin{proof} Denote by $\cW_m(\mu,\nu)$ the number of solutions of the inequality
$$\biggl|\sum_{j=m}^t\lambda_{s+j}(x_j^k-y_j^k)+\mu-\nu \biggr|<\delta,$$
with $P_j<x_j,y_j\le 2P_j$ $(m\le j\le t)$. Then as in the argument of the proof of Lemma \ref{lemma2.2}, one finds that
\begin{equation}\label{6.1}
\int_{-\infty}^\infty |\ff_1(\alpha)\cdots \ff_t(\alpha)H_\cZ (\alpha)|^2K_1(\alpha)\d\alpha \le \int_{\cZ^2} \delta^{-1}\cW_1(\mu,\nu)\d\mu \d\nu.
\end{equation}
We show by induction that when $\mu,\nu\in [N,N+M]$, then one has
\begin{equation}\label{6.2}
\cW_m(\mu,\nu)\le P_m\cdots P_tU_\delta (\mu-\nu)
\end{equation}
for $m=t,t-1,\ldots,1$.\par

The case $m=t$ follows from the argument of the proof of Lemma \ref{lemma2.2}, so we may now suppose that $1\le m<t$, and that
$$\cW_{m+1}(\mu,\nu)\le P_{m+1}\cdots P_tU_\delta (\mu-\nu).$$
If $\x,\y$ is a solution counted by $\cW_m (\mu,\nu)$ with $x_m\neq y_m$, then one has
$$|\lambda_{s+m}(x_m^k-y_m^k)|>k|\lambda_{s+m}|P_m^{k-1}=k|\lambda_{s+m}|c^{m+k}P_{m+1}^k.$$
Since we are at liberty to assume that $c$ is sufficiently large in terms of $t$, $k$ and $\L$, and $|\mu-\nu|\le P_t^{k-1}$, we arrive at the inequality
$$\biggl|\sum_{j=m+1}^t\lambda_{s+j}(x_j^k-y_j^k)+\mu-\nu \biggr|\le 2^k\lambda tP_{m+1}^k+P_t^{k-1}<|\lambda_{s+m}(x_m^k-y_m^k)|-\delta .$$
We are therefore forced to conclude that in fact $x_m=y_m$, and hence the inductive hypothesis gives
$$\cW_m(\mu,\nu)\le P_m\cW_{m+1}(\mu,\nu)\le P_m(P_{m+1}\cdots P_t)U_\delta (\mu-\nu).$$
We have therefore confirmed the inductive hypothesis (\ref{6.2}) for $1\le m\le t$.\par

Substituting the estimate (\ref{6.2}) with $m=1$ into (\ref{6.1}), we obtain
\begin{align*}\int_{-\infty}^\infty |\ff_1(\alpha)\cdots \ff_t(\alpha)H_\cZ (\alpha)|^2K_1(\alpha)\d\alpha &\le \delta^{-1}P_1\cdots P_t\int_\cZ\int_{\nu-\delta}^{\nu+\delta}\d\mu\d\nu \\
&=2P_1\cdots P_tZ.
\end{align*}
This completes the proof of the lemma.
\end{proof}

Next we record an analogue of Lemma \ref{lemma3.1}.

\begin{lemma}\label{lemma6.2} There exists a choice for the function $T(P)$, depending only on $\lambda_1$, $\lambda_2$ and $S(P)$, with the property that
$$\sup_{\alpha \in \fm}|g_1(\alpha)g_2(\alpha)| \ll P^2 T(P)^{-5^{-k}}.$$
\end{lemma}

\begin{proof} In view of our definition of the minor arcs $\fm$, the desired conclusion is immediate from \cite[Lemma 8.1]{W:FAF}.
\end{proof}

It follows from the conclusion of Lemma \ref{lemma6.2} that for all $\alpha\in \fm$, one has
\begin{equation}\label{6.3}
g_j(\alpha)\ll PT(P)^{-10^{-k}}
\end{equation}
for at least one suffix $j\in\{1,2\}$. We define $\fq_j$ to be the set of real numbers $\alpha \in \fm$ for which the upper bound (\ref{6.3}) holds, and then put $\fp_j=\fq_j\cup \ft$, so that $\fm\cup \ft\subseteq \fp_1\cup \fp_2$. In the interests of concision, we define $\upsilon_{ij}$ just as in the discussion of \S3 following (\ref{3.1}).

\par We must introduce some additional notation before announcing an auxiliary mean value estimate. When $1\le i\le s$, the set $\fB$ is measurable, and $K$ is integrable, we write
$$\cM^*_{i,t}(\fB;K)=\int_\fB|g_i(\alpha)|^tK(\alpha)\d\alpha .$$

\begin{lemma}\label{lemma6.3}
Suppose that $k\ge 3$, $t>\max\{2k+2,16\}$, $1\le i\le s$ and $j\in\{1,2\}$. Then for any fixed $\kappa>0$, one has
$$\cM^*_{i,t}(\fN_i\cap \fp_j;|K_\pm|)\ll_t P^{t-k}L(P)^{1-\kappa s\upsilon_{ij}}.$$
\end{lemma}

\begin{proof} We may suppose that $t=2k+2+2\gamma$, where $\gamma>0$. We put
$$u=2k+2+\gamma\quad \text{and}\quad A=6k(8k+\gamma)/\gamma.$$
Also, we define the function $\Psi(\alpha)$ by taking $\Psi(\alpha)=(q+P^k|q\alpha-a|)^{-1}$, when $\alpha \in \fN(q,a) \subseteq \fN$, and otherwise by putting $\Psi(\alpha)=0$. Then an application of \cite[Lemma 5.4]{PW:02} with $\cM=P^{3/4}$ and $T=2P^{1/4}$ shows that when $\alpha \in \fN(q,a)\subseteq \fN$, one has the estimate
$$g(\alpha)\ll P^{7/8+\ep}+ P(\log P)^3q^\ep \Psi(\alpha)^{1/(2k)}.$$
Next we observe that whenever $q+|q\alpha-a|P^k \le (\log P)^A$, then one may apply \cite[Lemma 8.5]{VW:91} to deduce that
$$g(\alpha)\ll q^\ep P(q+|q\alpha-a|P^k)^{-1/k}.$$
By combining these estimates, therefore, we conclude that when $\alpha \in \fN$ one has
$$g(\alpha)\ll P^{7/8+\ep}+P\Psi(\alpha)^{2/(4k+\gamma)}.$$

\par Next, since $\text{meas}(\fN)\ll P^{2-k}$, it follows by a change of variable that the last estimate delivers the bound
$$\int_{\fN_i\cap [0,1)}|g_i(\alpha)|^u\d\alpha \ll (P^{7/8+\ep})^uP^{2-k}+P^{2k+\gamma}\int_{\fN\cap [0,1)}|g(\beta)|^2\Psi(\beta)^{1+\gamma/(4k+\gamma)}\d\beta.$$
We have $\frac{7}{8}u+2<u$ whenever $u>16$, and thus the hypotheses of the statement of the lemma imply that the first term on the right hand side is $o(P^{u-k})$. On the other hand, a straightforward modification of the proof of \cite[Lemma 2]{Bru:88} (see Lemma \ref{lemma11.1} below) shows that
$$\int_{\fN \cap [0,1)}|g(\beta)|^2\Psi(\beta)^{1+\gamma/(4k+\gamma)}\d\beta \ll P^{2-k}.$$
We therefore conclude that
$$\int_{\fN_i\cap [0,1)}|g_i(\alpha)|^u\d\alpha \ll P^{u-k}.$$

\par We recall from (\ref{2.5}) that $K_\pm(\alpha)\ll L(P)\min\{1,\alpha^{-2}\}$. Then in view of the conclusion of Lemma \ref{lemma2.3} together with (\ref{6.3}), we have on the one hand
\begin{align*}
\int_{\fN_i\cap \fq_j}|g_i(\alpha)|^t|K_\pm(\alpha)|\d\alpha &\ll \Bigl(\sup_{\alpha\in\fq_j}|g_i(\alpha)|\Bigr)^\gamma L(P)\int_{\fN_i\cap [0,1)}|g_i(\alpha)|^u\d\alpha \\
&\ll \left(PT(P)^{-\upsilon_{ij}10^{-k}}\right)^\gamma L(P)P^{u-k}\\
&\ll P^{t-k}L(P)^{1-\kappa s\upsilon_{ij}},
\end{align*}
whilst on the other
\begin{align*}
\int_{\fN_i\cap \ft}|g_i(\alpha)|^t|K_\pm(\alpha)|\d\alpha &\ll g(0)^\gamma L(P)T(P)^{-1}\int_{\fN_i\cap [0,1)}|g_i(\alpha)|^u\d\alpha \\
&\ll (P^\gamma L(P)^{-\kappa s})P^{u-k}=P^{t-k}L(P)^{-\kappa s}.
\end{align*}
The conclusion of the lemma follows by combining these two upper bounds.
\end{proof}

\begin{lemma}\label{lemma6.4}
Suppose that $(s_0,\sigma_0)$ is a smooth accessible pair for $k$, and that $u$ is a real number with $u\ge s_0$. Then for $1\le i\le s$, there is a positive number $\omega$ with the property that
$$\cM^*_{i,u}(\fn_i;K_2^\pm)\ll P^{u-k-\omega -(u-s_0)\sigma_0+\ep}.$$
\end{lemma}

\begin{proof} On noting (\ref{2.8}), the desired estimate follows by inserting the hypothesised bounds (\ref{1.4}) into the conclusion of Lemma \ref{lemma2.3}.\end{proof}

Before coming to grips with the proof of Theorem \ref{theorem1.1}, we introduce some additional notation. Write
$$\cZ=\cZ_{s,k}^\tau(N,M;\L),\quad Z={\rm{meas}}(\cZ)\quad \text{and}\quad H(\alpha)=H_\cZ(\alpha).$$
We define
$$\tg(\alpha)=g_1(\alpha) \cdots g_s(\alpha), \quad \tff(\alpha)=\ff_1(\alpha) \cdots \ff_t(\alpha),$$
$$H^\dagger(\alpha)=\tff(\alpha)H(\alpha)\quad \text{and}\quad Z^\dagger=\tff(0)Z.$$
Also, we put
$$\cJ_1=\int_{-\infty}^\infty |\tff(\alpha)H(\alpha)|^2K_1(\alpha)\d\alpha ,$$
and when $\fB$ is measurable, we write
$$I_1(\fB)=\int_\fB|\tff(\alpha)\tg(\alpha)H(\alpha)K_-(\alpha)|\d\alpha.$$

As in the proof of Theorem \ref{theorem1.7}, by subdividing the interval $[N,N+M]$ into subintervals of length $P_t^{k-1}$, there is no loss of generality in supposing that $M=P_t^{k-1}$. If $\mu\in \cZ$, then it follows from (\ref{2.3}) that
\begin{equation}\label{6.4}
\int_{-\infty}^\infty \tff(\alpha)\tg(\alpha)e(-\alpha \mu)K_-(\alpha)\d\alpha=0.
\end{equation}
When $s\ge k+1$, moreover, the analysis leading to \cite[Lemma 9.4]{W:FAF} is easily modified to confirm that, uniformly in $\mu\in [N,N+M]$, one has
$$\int_\fM \tff(\alpha)\tg(\alpha)e(-\alpha \mu)K_-(\alpha)\d\alpha \gg \tau \tff(0)P^{s-k}.$$
Then by subtracting (\ref{6.4}) and integrating over all $\mu \in \cZ$, we deduce that
\begin{equation}\label{6.5}
I_1(\fp_1)+I_1(\fp_2)\gg \tau P^{s-k}Z^\dagger.
\end{equation}

\par Let $j$ be either $1$ or $2$. Then an application of H\"older's inequality reveals that
\begin{equation}\label{6.6}
I_1(\fp_j)\le \prod_{i=1}^s\Bigl( \int_{\fp_j}|g_i(\alpha)^sH^\dagger(\alpha)K_-(\alpha)|\d\alpha \Bigr)^{1/s}.
\end{equation}
Consider an index $i$ with $1\le i\le s$. In view of the hypotheses of the statement of Theorem \ref{theorem1.1}, we may suppose that $s>\max\{2k+2,16\}$, and thus we deduce from Lemma \ref{lemma6.3} that
\begin{align}
\int_{\fN_i\cap \fp_j}|g_i(\alpha)^sH^\dagger(\alpha)K_-(\alpha)|\d\alpha &\le H^\dagger(0)\cM_{i,s}^*(\fN_i\cap \fp_j;|K_-|)\notag \\
&\ll Z^\dagger P^{s-k}L(P)^{1-3s\upsilon_{ij}}.\label{6.7}
\end{align}
We may also suppose that $s\ge \tfrac{1}{2}s_0$, and so an application of Schwarz's inequality in combination with Lemmata \ref{lemma6.1} and \ref{lemma6.4} leads to the bound
\begin{align}
\int_{\fn_i}|g_i(\alpha)^sH^\dagger(\alpha)K_-(\alpha)|\d\alpha &\le \cJ_1^{1/2}\cM^*_{i,2s}(\fn_i;K_2^-)^{1/2}\notag \\
&\ll (Z^\dagger)^{1/2}(P^{2s-2k}\Xi_\ep)^{1/2},\label{6.8}
\end{align}
where we write
\begin{equation}\label{6.9}
\Xi_\ep=P^{k-\omega-(2s-s_0)\sigma_0+\ep}.
\end{equation}

\par On substituting (\ref{6.7}) and (\ref{6.8}) into (\ref{6.6}), we deduce that
\begin{align*}
I_1(\fp_1)+I_1(\fp_2)\ll &\,P^{s-k}\left(Z^\dagger L(P)+(Z^\dagger)^{1/2}\Xi_\ep^{1/2}\right)^{1-1/s}\\
&\,\times \left(Z^\dagger L(P)^{1-3s}+(Z^\dagger)^{1/2}\Xi_\ep^{1/2}\right)^{1/s}.
\end{align*}
Substituting this bound into (\ref{6.5}), we find that
$$P^{s-k}Z^\dagger\ll P^{s-k}L(P)^{-3}\left( Z^\dagger L(P)+P^\ep \Xi_\ep^{1/2}(Z^\dagger)^{1/2}\right).$$
Note that
$$\sum_{j=1}^t(1-1/k)^{j-1}=k-k(1-1/k)^t,$$
and hence $\tff(0) \gg P^kP_t^{1-k}=P^kM^{-1}$. Then on disentangling the last inequality, we conclude from (\ref{6.9}) that there exists a positive number $\Delta$ for which
$$P^kM^{-1}Z \ll Z^\dagger\ll P^\ep \Xi_\ep\ll P^{k-k\Delta -(2s-s_0)\sigma_0},$$
whence $Z\ll MP^{-k\Delta-(2s-s_0)\sigma_0}$. On recalling that $P=N^{1/k}$, the proof of Theorem \ref{theorem1.1} is complete.

\section{Solubility for smaller exponents}\label{sect:7}
For smaller values of $k$ we have available the option of applying a slightly different argument using some generating functions on a complete interval. This enables us to take advantage of the fact that Weyl's inequality is superior to the currently available values for $\sigma$ in (\ref{1.4}) when $k\le 6$. We also gain an advantage in our major arc treatment through the use of Lemma \ref{lemma3.2} instead of Lemma \ref{lemma6.3}. A full account of the details and ramifications of Theorem \ref{theorem1.2} would demand much more space than seems warranted for the present paper. We therefore aim for a concise exposition in which some of the less significant details are sketched rather than fully explained.\par

We begin by indicating how to establish the conclusion of Theorem \ref{theorem1.2} for $k\ge 7$. Recall the integers $s_0(k)$, $u_0(k)$ and $\sigma(k)$ recorded in Table 1. Then by reference to the tables of permissible exponents in \cite[\S\S9-22]{VW:00}, one finds from the last numerical value in each table that there is a positive integer $t<\tfrac{1}{2}s_0$ together with positive numbers $\lambda_t$, $\sigma$ and $\omega$, such that
\begin{align*}\int_{\fn\cap [0,1)}|g(\alpha)|^{s_0}\d\alpha&\le \Bigl( \sup_{\alpha \in \fn}|g(\alpha)|\Bigr)^{s_0-2t}\int_0^1|g(\alpha)|^{2t}\d\alpha \\
&\ll (P^{1-\sigma+\ep})^{s_0-2t}P^{\lambda_t+\ep}\ll P^{s_0-k-\omega}.
\end{align*}
For each integer $k$ with $7\le k\le 20$, therefore, one finds that the exponent pair $(s_0(k),\sigma(k))$ defined in Table 1, forms a smooth accessible pair for $k$. The analogous conclusion holds for larger values of $k$ by virtue of the results contained in \cite{W:92war} and \cite{W:95sws}. The conclusion of Theorem \ref{theorem1.2} therefore follows at once from Theorem \ref{theorem1.1} for $k\ge 7$, and it remains only to consider the exponents $k$ with $4\le k\le 6$.\par

When $4\le k\le 6$, define the integers $u=u(k)$, $v=v(k)$ and $w=w(k)$ as in the table below.
$$\boxed{\begin{matrix} k&4&5&6\\
u(k)&5&8&12\\
v(k)&12&20&26\\
w(k)&8&12&16\\
\end{matrix}}$$
\vskip.1cm
\begin{center}\text{Table 3: Parameters for the proof of Theorem \ref{theorem1.2}}\end{center}

\noindent We note for future reference that for each exponent $k$, one has
$$z(k)=(w-u)/(1-u/v)>k+1.$$
On considering the underlying Diophantine equations, the methods of \cite{Va:89} and \cite{VW:95} show that for each $k$ one has
\begin{equation}\label{7.1}
\int_0^1|g(\alpha)|^v\d\alpha \ll P^{v-k}.
\end{equation}
This argument is made explicit for $k=4$ in \cite[Lemma 5.2]{Va:89} and for $k=5$ in \cite[Lemma 7.3]{VW:95}. The reader should not experience any difficulty in extracting the analogous conclusion for $k=6$ in like manner. Finally, we note that from the tables of exponents in \cite{BW:00-27} and \cite{VW:95}, one has
\begin{equation}\label{7.2}
\int_0^1|g(\alpha)|^{2u}\d\alpha \ll P^{2u-k+\theta+\ep},
\end{equation}
where $\theta=\theta_{u,k}$ is given by
$$\theta_{5,4} = 0.213431,\quad \theta_{8,5} = 0.077363,\quad \mbox{and}\quad \theta_{12,6}=0.$$

\par We suppose that $s\ge w=\tfrac{1}{2}(s_0+u_0)$, and employ the notation introduced in \S6, modifying the definition of $\tg(\alpha)$ by putting
$$\tg(\alpha)=g_1(\alpha)\cdots g_u(\alpha)f_{u+1}(\alpha)\cdots f_s(\alpha).$$
As in the proof of Theorem \ref{theorem1.7}, by subdividing the interval $[N,N+M]$ into subintervals of length $P_t^{k-1}$, there is no loss of generality in supposing that $M=P_t^{k-1}$. When $s\ge k+1$, moreover, the analysis leading to \cite[Lemma 9.4]{W:FAF} is again easily modified to confirm that, uniformly in $\mu\in [N,N+M]$, one has
$$\int_\fM \tff(\alpha)\tg(\alpha)e(-\alpha \mu)K_-(\alpha)\d\alpha \gg \tau \tff(0)P^{s-k}.$$
Then we deduce as in the argument of the proof of Theorem \ref{theorem1.1} that the lower bound (\ref{6.5}) remains valid in the present circumstances.\par

Let $j$ be either $1$ or $2$. Then an application of H\"older's inequality reveals that
\begin{equation}\label{7.3}
I_1(\fp_j)\le \prod_{i=1}^u\prod_{m=u+1}^s\Bigl( \int_{\fp_j}|g_i(\alpha)^uf_m(\alpha)^{s-u}H^\dagger(\alpha)K_-(\alpha)|\d\alpha \Bigr)^{1/(u(s-u))}.
\end{equation}
Consider indices $i$  and $m$ with $1\le i\le u<m\le s$. An application of H\"older's inequality yields the bound
\begin{align*}
\int_{\fN_m\cap \fp_j}|g_i(\alpha)^u&f_m(\alpha)^{s-u}H^\dagger(\alpha)K_-(\alpha)|\d\alpha \\
&\le H^\dagger(0)\cM_{i,v}^*(\R;|K_-|)^{u/v}\cM_{m,y}(\fN_m\cap \fp_j;|K_-|)^{1-u/v},
\end{align*}
where
$$y=(s-u)/(1-u/v)\ge (w-u)/(1-u/v)=z.$$
Since we may suppose that $z>k+1$, we deduce from (\ref{7.1}) via Lemma \ref{lemma2.3} in combination with Lemma \ref{lemma3.2} that
\begin{equation}\label{7.4}
\int_{\fN_m\cap \fp_j}|g_i(\alpha)^uf_m(\alpha)^{s-u}H^\dagger(\alpha)K_-(\alpha)|\d\alpha \ll Z^\dagger P^{s-k}L(P)^{1-3s\upsilon_{ij}}.
\end{equation}

\par Next, by applying the mean value estimate (\ref{7.2}) together with Lemma \ref{lemma2.3}, we find that
$$\cM_{i,2u}^*(\R;K_2^-)\ll P^{2u-k+\theta+\ep}.$$
Then on recalling Weyl's inequality (see \cite[Lemma 2.4]{V:HL}), if follows from an application of Schwarz's inequality in combination with Lemma \ref{lemma6.1} that
\begin{align*}
\int_{\fn_m}|g_i(\alpha)^uf_m(\alpha)^{s-u}H^\dagger(\alpha)K_-(\alpha)|\d\alpha &\le \left( \sup_{\alpha \in \fn_m}|f_m(\alpha)|\right)^{s-u}\cJ_1^{1/2}\cM^*_{i,2u}(\R;K_2^-)^{1/2}\\
&\ll (Z^\dagger)^{1/2}(P^{1-\sigma+\ep})^{s-u}\left( P^{2u-k+\theta+\ep}\right)^{1/2}.
\end{align*}
Since $\sigma=2^{1-k}$, we see that there exists a positive number $\Delta$ with the property that
\begin{equation}\label{7.5}
\int_{\fn_m}|g_i(\alpha)^uf_m(\alpha)^{s-u}H^\dagger(\alpha)K_-(\alpha)|\d\alpha \ll (Z^\dagger)^{1/2}(P^{2s-2k}\Xi
)^{1/2},
\end{equation}
where we write
$$\Xi=P^{k-k\Delta -(2s-s_0)\sigma}.$$

\par On substituting (\ref{7.4}) and (\ref{7.5}) into (\ref{7.3}) and from there into (\ref{6.5}), we deduce that
$$P^{s-k}Z^\dagger \ll P^{s-k}\left(Z^\dagger L(P)+(Z^\dagger)^{1/2}\Xi^{1/2}\right)^{1-1/s}\left(Z^\dagger L(P)^{1-3s}+(Z^\dagger)^{1/2}\Xi^{1/2}\right)^{1/s}.$$
Then on disentangling the last inequality, we conclude just as before that
$$P^kM^{-1}Z\ll Z^\dagger\ll P^{k-k\Delta -(2s-s_0)\sigma},$$
whence $Z\ll MP^{-k\Delta-(2s-s_0)\sigma}$. On recalling that $P=N^{1/k}$, the proof of the main conclusion of Theorem \ref{theorem1.2} is complete.\vskip.3cm

We now turn to a brief sketch of the proof of the estimate $Z_{s+t,k}(N,M)\ll MN^{-\Delta}$, for a sufficiently small $\Delta>0$, valid for $s\ge \tfrac{1}{2}s_0$. Here we must obtain greater control of the behaviour of the generating functions $g_i(\alpha)$ on the major arcs $\fN_i$. This we achieve by means of two modifications to the above argument. First we replace the underlying sequence $\cA(P,R)$ by the related sequence
$$\cC(P,R)=\{lm:\text{$1\le l\le \sqrt{R}$, $1\le m\le P/\sqrt{R}$, $p|m\Rightarrow \sqrt{R}<p\le R$}\}.$$
We also replace the major arc $\fN(q,a)$ by the set of real numbers $\alpha$ for which $|q\alpha -a|\le \sqrt{R}P^{-k}$, and then write $\fN$ for the union of the intervals $\fN(q,a)$ over all co-prime integers $a$ and $q$ with $1\le q\le \sqrt{R}$. We then put $\fn=\R\setminus \fN$. As the reader will find in the papers \cite{BW:01} and \cite{BW:01a}, with such a modification, the generating function $g(\alpha)$ may be analysed on $\fN$ essentially as precisely as the generating function $f(\alpha)$. Thus one finds that the conclusion of Lemma \ref{lemma6.3} holds with the hypothesis on $t$ weakened to the condition $t>k+1$. On the other hand, the conclusion of Lemma \ref{lemma6.4} now holds only with $\sigma_0$ replaced by a positive number depending on $\eta$, which we recall defines $R$ by means of the relation $R=P^\eta$.\par

With the modifications described in the previous paragraph, the argument of \S6 employed in the proof of Theorem \ref{theorem1.1} applies without further modification to show that when $s\ge \max \{k+1,\tfrac{1}{2}s_0\}$, then for some positive number $\Delta$ one has
$$Z_{s+t,k}(N,M)\ll MN^{-\Delta}.$$
This completes our sketch of the proof of the remaining part of Theorem \ref{theorem1.2}, the details, though not difficult, being lengthy to record in full.

\section{Cubic forms}\label{sect:8}
In this section we outline the proof of Theorem \ref{theorem1.3}. Although in principle one has only to apply the methods of \S7 together with Lemma \ref{lemma4.1} to establish the desired bounds for $Z_{s,3}(N)$ $(s=5,6)$, the small number of  available variables leads to complications in handling the contributions from the major arcs. We note that the bound $Z_{4,3}(N)=o(N)$ is a consequence of Theorem \ref{theorem1.5}, which we have already established. Thus it suffices to consider $Z_{s,3}(N)$ for $s=5$ and $6$.\par

We begin by establishing an auxiliary mean value estimate. We suppose throughout that $k=3$ and $P=N^{1/3}$. Also, we let $\fV(q,a)$ denote the set of real numbers $\alpha \in [0,1)$ with $|q\alpha -a|\le P^{-9/4}$, and write $\fV$ for the union of the sets $\fV(q,a)$ with $0\le a\le q\le P^{3/4}$ and $(a,q)=1$. We then put $\fv=[0,1)\setminus \fV$.

\begin{lemma}\label{lemma8.1} Suppose that $\cZ\subseteq [0,N]$ has measure $Z$, and $|\eta_\mu|=1$ for all $\mu\in \cZ$. Then whenever $\lambda$ is a non-zero real number, one has
$$\int_{-\infty}^\infty |g(\lambda \alpha)^4H_{\eta,\cZ}(\alpha)^2|K_1(\alpha)\d\alpha \ll P^3(\log P)^{2+\ep}Z+P(\log P)^\ep Z^2.$$
\end{lemma}

\begin{proof} As in the proof of Lemma \ref{lemma2.1}, the mean value under consideration may be written as
$$I=\int_{\cZ^2}{\overline \eta}_\mu\eta_\nu\int_{-\infty}^\infty |g(\lambda \alpha )|^4e(\alpha (\mu-\nu))K_1(\alpha )\d\alpha \d\mu\d\nu .$$
The relation (\ref{2.9}) shows that
\begin{equation}\label{8.1}
I\le \int_{\cZ^2}\delta^{-1}\cW(\mu,\nu)\d\mu \d\nu,
\end{equation}
where
$$\cW(\mu,\nu)=\sum_{1\le x_1,\ldots ,x_4\le P}\max\{0,1-\delta^{-1}|\lambda (x_1^3+x_2^3-x_3^3-x_4^3)+\mu-\nu|\}.$$
We observe that on writing $R(h)$ for the number of integral solutions of the equation
$$x_1^3+x_2^3-x_3^3-x_4^3=h,$$
with $1\le x_i\le P$ $(1\le i\le 4)$, then one has
\begin{equation}\label{8.2}
\cW(\mu,\nu)=\sum_{h\in \Z}R(h)\max\{0,1-\delta^{-1}|\lambda h+\mu-\nu|\}.
\end{equation}

\par When $\fB\subseteq [0,1)$ is measurable, write
$$R(h;\fB)=\int_\fB|f(\alpha)|^4e(-h\alpha)\d\alpha .$$
We analyse $R(h)$ by means of the Hardy-Littlewood method, and in this way we deduce from (\ref{8.1}) and (\ref{8.2}) that
\begin{equation}\label{8.3}
\int_{-\infty}^\infty |g(\lambda \alpha )^4H_{\eta,\cZ}(\alpha)^2|K_1(\alpha)\d\alpha \ll T(\fV)+T(\fv),
\end{equation}
where we write
$$T(\fB)=\delta^{-1}\int_{\cZ^2}\sum_{h\in\Z}R(h;\fB)\max\{0,1-\delta^{-1}|\lambda h+\mu-\nu|\}\d\mu\d\nu .$$

\par We begin by analysing the major arc contribution. Although there is no convenient source in the literature, the reader will experience no difficulty in applying the methods of \cite[\S\S4.3 and 4.4]{V:HL} to confirm that
$$R(h;\fV)\ll P(\log P)^\ep\quad (h\ne 0),\quad \text{and}\quad R(0;\fV)\ll P^{1+\ep}.$$
We note that a precise form of the first of these bounds may be found in \cite[equation (1.3)]{Kaw:96}. Equipped with these bounds, we deduce that
\begin{align*}
T(\fV)\ll &\,\delta^{-1}P^{1+\ep}\int_{\cZ^2}U_\delta (\mu-\nu)\d\mu\d\nu \\
&\,+\delta^{-1}P(\log P)^\ep \int_{\cZ^2}\sum_{h\in\Z}U_\delta (\lambda h+\mu-\nu)\d\mu\d\nu\\
&\ll \delta^{-1}P^{1+\ep}\int_\cZ\int_{\nu-\delta}^{\nu+\delta}\d\mu\d\nu+\delta^{-1}P(\log P)^\ep\int_{\cZ^2}\d\mu\d\nu .
\end{align*}
Consequently, one has
\begin{equation}\label{8.4}
T(\fV)\ll P^{1+\ep}Z+P(\log P)^\ep Z^2.
\end{equation}

\par Next, we recall that an enhanced version of Weyl's inequality is available from the argument underlying the proof of \cite[Lemma 1]{Va:86c}, on applying bounds of Hall and Tenenbaum \cite{HT:86} for Hooley's $\Delta$-function in place of Vaughan's application of Hooley \cite{Hoo:79}. Thus one finds that
\begin{equation}\label{8.5}
\sup_{\alpha \in \fv}|f(\alpha)|\ll P^{3/4}(\log P)^{1/4+\ep}.
\end{equation}
Observe next that
$$T(\fv)=\sum_{|h|\le (N+\delta)/\lambda }R(h;\fv)\int_{-\infty}^\infty |H_\cZ(\beta)|^2e(\lambda h\beta)K_1(\beta)\d\beta .$$
But equipped with the estimate (\ref{8.5}), we deduce that
\begin{align*}
\sum_{|h|\le (N+\delta)/\lambda}R(h;\fv)e(\lambda h\beta)&=\int_\fv |f(\alpha)|^4\sum_{|h|\le (N+\delta)/\lambda}e(h(\lambda \beta -\alpha))\d\alpha \\
&\ll \int_\fv |f(\alpha )|^4\min\{(N+\delta)/\lambda ,\|\alpha -\lambda \beta\|^{-1}\}\d\alpha \\
&\ll P^3(\log P)^{1+\ep}\int_0^1\min \{ (N+\delta)/\lambda ,\|\alpha -\lambda \beta\|^{-1}\}\d\alpha .
\end{align*}
We therefore obtain the bound
$$\sum_{|h|\le (N+\delta)/\lambda}R(h;\fv)e(\lambda h\beta)\ll P^3(\log P)^{2+\ep},$$
and hence we conclude from Lemma \ref{lemma2.1} that
\begin{equation}\label{8.6}
T(\fv)\ll P^3(\log P)^{2+\ep}\int_{-\infty}^\infty |H_\cZ(\beta)|^2K_1(\beta)\d\beta \ll P^3(\log P)^{2+\ep}Z.\end{equation}

\par The proof of the lemma is completed by substituting the estimates (\ref{8.4}) and (\ref{8.6}) into (\ref{8.3}).
\end{proof}

We also require a variant of Lemma \ref{lemma6.3} providing a reasonably sharp bound with relatively few variables. We must first introduce some auxiliary sets of major arcs. When $1\le i\le 5$, we write $\fV_i$ for the set of real numbers $\alpha \in \fm\cup \ft$ for which $\lambda_i\alpha \in \fV\pmod{1}$, and $\fv_i=(\fm\cup \ft)\setminus \fV_i$. In addition, we define $\fW(q,a)$ to be the set of $\alpha\in [0,1)$ such that $|q\alpha -a|\le (\log P)^{100}P^{-3}$, and take $\fW$ to be the union of the sets $\fW(q,a)$ with $0\le a\le q\le (\log P)^{100}$ and $(a,q)=1$. We then put $\fw=[0,1)\setminus \fW$. Also, when $1\le i\le 5$, we write $\fW_i$ for the set of real numbers $\alpha \in \fm\cup \ft$ for which $\lambda_i\alpha \in \fW\pmod{1}$, and $\fw_i=(\fm\cup \ft)\setminus \fW_i$.\par

We pause to record an auxiliary lemma.

\begin{lemma}\label{lemma8.2}
Suppose that $t>5$ and $1\le i\le s$. Then one has
$$\cM^*_{i,t}(\fW_i;|K_\pm|)\ll_t P^{t-3}L(P).$$
\end{lemma}

\begin{proof} We may suppose that $t=2+(3+\gamma)(1+\gamma)$, where $\gamma>0$. We define the function $\Psi(\alpha)$ by taking $\Psi(\alpha)=(q+P^3|q\alpha-a|)^{-1}$ when $\alpha \in \fW(q,a) \subseteq \fW$, and otherwise by putting $\Psi(\alpha)=0$. One may apply \cite[Lemma 8.5]{VW:91} to deduce that when $\alpha \in \fW(q,a)\subseteq \fW$, then
$$g(\alpha)\ll q^\ep P(q+|q\alpha-a|P^3)^{-1/3}.$$
Then when $\alpha \in \fW$ one has
$$g(\alpha)\ll P\Psi(\alpha)^{1/(3+\gamma)}.$$
It follows by a change of variable that the last estimate delivers the bound
$$\int_{\fW_i\cap [0,1)}|g_i(\alpha)|^t\d\alpha \ll P^{t-2}\int_\fW|g(\beta)|^2\Psi(\beta)^{1+\gamma}\d\beta.$$
A straightforward modification of the proof of \cite[Lemma 2]{Bru:88} (see Lemma \ref{lemma11.1} below) shows that
$$\int_\fW|g(\beta)|^2\Psi(\beta)^{1+\gamma}\d\beta \ll P^{-1}.$$
We therefore arrive at the upper bound
$$\int_{\fW_i\cap [0,1)}|g_i(\alpha)|^t\d\alpha \ll P^{t-3}.$$
The conclusion of the lemma consequently follows from Lemma \ref{lemma2.3}.
\end{proof}

Now we initiate the proof of Theorem \ref{theorem1.3}. We write $\cZ=\cZ_{s,3}^\tau (N;\L)$, and put $Z_{s,3}(N)=\text{meas}(\cZ)$. We seek first to bound $Z=Z_{5,3}(N)$, and fix $s=5$ in this first case. Write
$$\tg(\alpha)=f_1(\alpha)f_2(\alpha)g_3(\alpha)g_4(\alpha)g_5(\alpha),$$
and then put
\begin{equation}\label{8.7}
I(\fB)=\int_\fB |\tg(\alpha)H(\alpha)K_-(\alpha)|\d\alpha,
\end{equation}
in which we have written $H(\alpha)=H_\cZ(\alpha)$. We also write
$$\cJ_m=\int_{-\infty}^\infty |g_m(\alpha)^4H(\alpha)^2|K_1(\alpha)\d\alpha .$$
The analysis leading to \cite[Lemma 9.4]{W:FAF} is easily modified to confirm that, uniformly in $\mu\in (N/2,N]$, one has
$$\int_\fM \tg(\alpha)e(-\alpha \mu)K_-(\alpha)\d\alpha \gg \tau P^2.$$
Moreover, for each $\mu \in \cZ$ one has
$$\int_{-\infty}^\infty \tg(\alpha)e(-\alpha \mu)K_-(\alpha)\d\alpha =0.$$
Then by subtracting and integrating over $\mu \in \cZ$, we deduce as before that
\begin{equation}\label{8.8}
I(\fp_1)+I(\fp_2)\gg P^2Z.
\end{equation}

\par Let $j$ be either $1$ or $2$. Then an application of H\"older's inequality to (\ref{8.7}) reveals that
$$I(\fp_j)\le \prod_{i=1}^2\prod_{m=1}^3\Bigl( \int_{\fp_j}|f_i(\alpha)^2g_m(\alpha)^3H(\alpha)K_-(\alpha)|\d\alpha\Bigr)^{1/6}.$$
Consider indices $i$ and $m$ with $i\in\{1,2\}$ and $m\in\{3,4,5\}$. By applying H\"older's inequality, we find that
\begin{align*}
\int_{\fV_i\cap\fW_m\cap\fp_j}&|f_i(\alpha)^2g_m(\alpha)^3H(\alpha)K_-(\alpha)|\d\alpha \\
&\le H(0)\cM_{i,22/5}(\fV_i\cap\fp_j;|K_-|)^{5/11}\cM^*_{m,11/2}(\fW_m;|K_-|)^{6/11}.
\end{align*}
From Lemmata \ref{lemma3.2} and \ref{lemma8.2}, therefore, we deduce that
\begin{align}
\int_{\fV_i\cap\fW_m\cap\fp_j}|f_i(\alpha)^2&g_m(\alpha)^3H(\alpha)K_-(\alpha)|\d\alpha \notag\\
&\le H(0)(P^{7/5}L(P)^{1-9s\upsilon_{ij}})^{5/11}(P^{5/2}L(P))^{6/11}\notag\\
&\ll ZP^2L(P)^{1-3s\upsilon_{ij}}.\label{8.9}
\end{align}

\par Next, an application of H\"older's inequality reveals that
\begin{align*}
\int_{\fV_i\cap\fw_m}&|f_i(\alpha)^2g_m(\alpha)^3H(\alpha)K_-(\alpha)|\d\alpha \\
&\le \Bigl( \sup_{\alpha \in \fw_m}|g_m(\alpha)|\Bigr)^{7/11}\cM_{i,22/5}(\fV_i\cap\fp_j;K_2^-)^{5/11}\cJ_m^{1/2}\cM^*_{m,8}(\R;K_2^-)^{1/22}.
\end{align*}
On the one hand, we observe that as a consequence of \cite[Lemmata 7.2 and 8.5]{VW:91} together with \cite[Theorem 1.8]{Va:89}, one has
$$\sup_{\alpha \in \fw_m}|g_m(\alpha)|\ll P(\log P)^{-30}.$$
On the other hand, by considering the underlying Diophantine equations, it follows from Vaughan \cite[Theorem 2]{Va:86c} that
$$\int_0^1|g(\beta)|^8\d\beta \ll P^5,$$
whence from Lemma \ref{lemma2.3} we obtain
\begin{equation}\label{8.10}
\cM^*_{m,8}(\R;K_2^-)\ll P^5.
\end{equation}
We therefore deduce from Lemmata \ref{lemma3.2} and \ref{lemma8.1} that
\begin{align*}
\int_{\fV_i\cap\fw_m}|f_i(\alpha)^2g_m(\alpha)^3&H(\alpha)K_-(\alpha)|\d\alpha \\
\ll &\,(P(\log P)^{-30})^{7/11}(P^{7/5})^{5/11}(P^5)^{1/22}\\
&\times (P^3(\log P)^{2+\ep}Z+P(\log P)^\ep Z^2)^{1/2},
\end{align*}
whence
\begin{equation}\label{8.11}
\int_{\fV_i\cap\fw_m}|f_i(\alpha)^2g_m(\alpha)^3H(\alpha)K_-(\alpha)|\d\alpha \ll P^3Z^{1/2}+P^2(\log P)^{-2}Z.
\end{equation}

\par Finally, an application of Schwarz's inequality in combination with Lemma \ref{lemma2.3} delivers the bound
$$\int_{\fv_i}|f_i(\alpha)^2g_m(\alpha)^3H(\alpha)K_-(\alpha)|\d\alpha \ll \Bigl( \sup_{\alpha \in \fv}|f(\alpha)|\Bigr)^2\Bigl(\int_0^1|g(\alpha)|^6 \d\alpha \Bigr)^{1/2}\cJ^{1/2},$$
where $\cJ$ is defined as in (\ref{3.11}). Define $\theta$ by means of the relation $\theta^{-1}=852+16\sqrt{2833}$. Then, by utilising \cite[Theorem 1.2]{W:3cubes} together with (\ref{8.5}) and Lemma \ref{lemma2.1}, we find that
\begin{align}
\int_{\fv_i}|f_i(\alpha)^2g_m(\alpha)^3H(\alpha)K_-(\alpha)|\d\alpha &\ll (P^{3/4+\ep})^2(P^{13/4-\theta+\ep})^{1/2}Z^{1/2}\notag \\
&\ll P^{25/8-\theta/2+3\ep}Z^{1/2}.\label{8.12}
\end{align}

\par By combining (\ref{8.9}), (\ref{8.11}) and (\ref{8.12}), we reach the bound
$$\int_{\fp_j}|f_i(\alpha)^2g_m(\alpha)^3H(\alpha)K_-(\alpha)|\d\alpha \ll P^2L(P)^{1-3s\upsilon_{ij}}Z+P^{25/8-\theta/2+\ep}Z^{1/2}.$$
We therefore deduce that
\begin{align*}
I(\fp_1)+I(\fp_2)\ll &\,\Bigl( ZP^2L(P)+P^{25/8-\theta/2+\ep}Z^{1/2}\Bigr)^{1/2}\\
&\, \times \Bigl( ZP^2L(P)^{-6}+P^{25/8-\theta/2+\ep}Z^{1/2}\Bigr)^{1/2},
\end{align*}
so that we may conclude from (\ref{8.8}) that
$$ZP^2\ll ZP^2L(P)^{-1}+P^{25/8-\theta/2+\ep}Z^{1/2}+Z^{3/4}P^{41/16-\theta/4+\ep}.$$
Disentangling this inequality, we arrive at the bound
$$Z\ll P^{9/4-\theta+\ep}\ll N^{3/4-\theta/3+\ep}.$$
This completes the proof of the estimate for $Z_{5,3}(N)$ asserted in Theorem \ref{theorem1.3}.\vskip.3cm

We now turn to the bound for $Z_{6,3}(N)$ recorded in Theorem \ref{theorem1.3}. On this occasion we write
$$\tg(\alpha)=f_1(\alpha)f_2(\alpha)f_3(\alpha)g_4(\alpha)g_5(\alpha)g_6(\alpha),$$
and then define $I(\fB)$ as in (\ref{8.7}). We also write
$$\cJ_i=\int_{-\infty}^\infty |f_i(\alpha)^2H(\alpha)^2|K_1(\alpha)\d\alpha .$$
The analysis leading to \cite[Lemma 9.4]{W:FAF} is easily modified to confirm that, uniformly in $\mu\in (N/2,N]$, one has
$$\int_\fM \tg(\alpha)e(-\alpha \mu)K_-(\alpha)\d\alpha \gg \tau P^3.$$
Then we deduce as before that
\begin{equation}\label{8.13}
I(\fp_1)+I(\fp_2)\gg P^3Z.
\end{equation}

\par Let $j$ be either $1$ or $2$. Then an application of H\"older's inequality reveals that
\begin{equation}\label{8.14}
I(\fp_j)\le \prod_{i=1}^3\prod_{m=1}^3\Bigl( \int_{\fp_j}|f_i(\alpha)^3g_m(\alpha)^3H(\alpha)K_-(\alpha)|\d\alpha\Bigr)^{1/9}.
\end{equation}
Consider indices $i$ and $m$ with $i\in\{1,2,3\}$ and $m\in\{4,5,6\}$. By applying H\"older's inequality, we find that
\begin{align*}
\int_{\fV_i\cap\fp_j}&|f_i(\alpha)^3g_m(\alpha)^3H(\alpha)K_-(\alpha)|\d\alpha \\
&\le H(0)\cM_{i,24/5}(\fV_i\cap\fp_j;|K_-|)^{5/8}\cM^*_{m,8}(\R;|K_-|)^{3/8}.
\end{align*}
By employing an argument akin to that delivering the bound (\ref{8.10}), one finds that $\cM^*_{m,8}(\R;K^-)\ll P^5L(P)$. Then from Lemma \ref{lemma3.2}, we deduce that
\begin{align}
\int_{\fV_i\cap\fp_j}&|f_i(\alpha)^3g_m(\alpha)^3H(\alpha)K_-(\alpha)|\d\alpha \notag\\
&\le H(0)(P^{9/5}L(P)^{1-9s\upsilon_{ij}})^{5/8}(P^5L(P))^{3/8}\ll ZP^3L(P)^{1-3s\upsilon_{ij}}.\label{8.15}
\end{align}

\par Next, an application of Schwarz's inequality in combination with Lemma \ref{lemma2.3} reveals that
$$\int_{\fv_i}|f_i(\alpha)^3g_m(\alpha)^3H(\alpha)K_-(\alpha)|\d\alpha \ll \Bigl( \sup_{\alpha \in \fv}|f(\alpha)|\Bigr)^2\Bigl(\int_0^1|g(\alpha)|^6 \d\alpha \Bigr)^{1/2}\cJ_i^{1/2}.$$
Then as a consequence of \cite[Theorem 1.2]{W:3cubes} together with Lemma \ref{lemma4.1} and (\ref{8.5}), we find that
\begin{align}
\int_{\fv_i}|f_i(\alpha)^3g_m(\alpha)^3&H(\alpha)K_-(\alpha)|\d\alpha \notag \\
&\ll (P^{3/4+\ep})^2(P^{13/4-\theta+\ep})^{1/2}(PZ+P^{1/2+\ep}Z^{3/2})^{1/2}\notag \\
&\ll P^{29/8-\theta/2+3\ep}Z^{1/2}+P^{27/8-\theta/2+3\ep}Z^{3/4}.\label{8.16}
\end{align}

\par Combining (\ref{8.15}) and (\ref{8.16}), we obtain the bound
\begin{align*}
\int_{\fp_j}&|f_i(\alpha)^3g_m(\alpha)^3H(\alpha)K_-(\alpha)|\d\alpha \\
&\ll P^3L(P)^{1-3s\upsilon_{ij}}Z+P^{29/8-\theta/2+\ep}Z^{1/2}+P^{27/8-\theta/2+\ep}Z^{3/4},
\end{align*}
and hence by (\ref{8.14}) we have
\begin{align*}
I(\fp_1)+I(\fp_2)\ll &\, \Bigl( P^3L(P)Z+P^{29/8-\theta/2+\ep}Z^{1/2}+P^{27/8-\theta/2+\ep}Z^{3/4}\Bigr)^{2/3}\\
&\, \
\times \Bigl( ZP^3L(P)^{-6}+P^{29/8-\theta/2+\ep}Z^{1/2}+P^{27/8-\theta/2+\ep}Z^{3/4}\Bigr)^{1/3}.
\end{align*}
We may therefore conclude from (\ref{8.13}) that
$$ZP^3\ll ZP^3L(P)^{-1}+P^{29/8-\theta/2+\ep}Z^{1/2}+P^{27/8-\theta/2+\ep}Z^{3/4}.$$
Disentangling this inequality, we arrive at the bound
$$Z\ll P^{5/4-\theta+\ep}+P^{3/2-2\theta+\ep}\ll N^{1/2-2\theta/3+\ep}.$$
This completes the proof of the estimate for $Z_{6,3}(N)$ asserted in Theorem \ref{theorem1.3}.\vskip.3cm

\section{Lower bound theorems}\label{sect:9}
The proof of Theorem \ref{theorem1.8} is a very simple argument based on use of the kernel
\begin{equation}\label{9.1}
K(\alpha)=\left(\frac{\sin \pi \alpha \tau}{\pi \alpha}\right)^{\! \! 2},
\end{equation}
which is obtained by setting $\delta=\tau$ in the definition of $\delta^2K_1(\alpha)$ given by (\ref{2.7}). Let $\cY=\cY^\tau_{s,k}(N;\L)$, and write $Y={\rm meas}(\cY)$ and $H(\alpha)=H_\cY(\alpha)$. We have
$$\int_{-\infty}^\infty g_1(\alpha)\cdots g_s(\alpha)H(\alpha)K(\alpha)\d\alpha =\sum_{\x \in \cA(P,R)^s} \int_\cY \hK(\lambda_1x_1^k+\dots +\lambda_sx_s^k-\mu)\d\mu,$$
and it follows from (\ref{2.9}) that $\hK(\beta)=\max\{0,\tau-|\beta|\}$. Hence for each $\x$ there is an interval $[F(\x)-\tau/2,F(\x)+\tau/2]$ contained in $\cY$ on which $\hK(F(\x)-\mu)\ge \tau/2$, and it follows that
$$\int_{-\infty}^\infty g_1(\alpha)\cdots g_s(\alpha)H(\alpha)K(\alpha)\d\alpha \gg \tau^2P^s.$$
Applying H\"older's inequality, we therefore obtain
$$\tau^2P^s\ll \left(\int_{-\infty}^\infty |H(\alpha)|^2K(\alpha)\d\alpha\right)^{\! 1/2}\prod_{j=1}^s\left(\int_{-\infty}^\infty |g_j(\alpha)|^{2s}K(\alpha)\d\alpha\right)^{\! 1/(2s)}.$$
Thus if $\Delta=\Delta_{s,k}$ is an admissible exponent, then we deduce from (\ref{1.9}), (\ref{2.7}), Lemma \ref{lemma2.1}, and Lemma \ref{lemma2.3} that
$$\tau^2P^s\ll \tau Y^{1/2}(P^{2s-k+\Delta+\ep})^{1/2},$$
and hence $Y\gg \tau^2P^{k-\Delta-\ep}$. In particular, we deduce from \cite[Theorem 1.2]{W:3cubes} that the exponent $\Delta_{3,3}=(\sqrt{2833}-43)/41=0.249413\ldots $ is admissible, and this completes the
proof of Theorem \ref{theorem1.8}.

\section{Linear combinations of two primes}\label{sect:10}
In view of the analysis of \cite{P:02lfp} and the pattern of argument established in the previous sections, we can be somewhat brief in our proof of Theorem \ref{primes}. Here we employ the weighted exponential sums
$$h_i(\alpha)=\sum_{p \le X}(\log p)e(\lambda_ip\alpha),$$
and we write $\cZ=\cZ^*(X;\L,\tau)\cap (X/2,X]$ and $Z={\rm{meas}}(\cZ)$. We let $\fM$ denote the set of $\alpha$ satisfying $|\alpha|\le (\log X)^AX^{-1}$, for a suitable constant $A>0$, and we define the boundary between the minor and trivial arcs using a suitable function $T(X)$, with the property that
\begin{equation}\label{10.1}
\sup_{\alpha \in \fm} |h_1(\alpha) h_2(\alpha)|\ll X^2 L(X)^{-1},
\end{equation}
wherein the function $L(X)\ll T(X)$ tends to infinity sufficiently slowly. The existence of such functions $T$ and $L$ follows from \cite{P:02lfp} (see Lemma 2 and the discussion in \S3). We let $K(\alpha)$ be as in
(\ref{9.1}) and write $H(\alpha)=H_\cZ(\alpha)$. If $\mu\in \cZ$, then one has
$$\int_{-\infty}^\infty h_1(\alpha)h_2(\alpha)e(-\alpha \mu)K(\alpha)\d\alpha =0,$$
and the analysis of \cite[\S4]{P:02lfp} yields
$$\int_\fM h_1(\alpha)h_2(\alpha)e(-\alpha \mu)K(\alpha)\d\alpha \gg \tau^2 \mu.$$
Hence on integrating over $\cZ$ and applying Schwarz's inequality, we find that
\begin{equation}\label{10.2}
\tau^2XZ\ll \cI_1^{1/2}\cI_2^{1/2},
\end{equation}
where
$$\cI_1=\int_{-\infty}^\infty |H(\alpha)|^2K(\alpha)\d\alpha \quad \mbox{and}\quad \cI_2=\int_{\fm \cup \ft}|h_1(\alpha)h_2(\alpha)|^2K(\alpha)\d\alpha.$$
By Lemma \ref{lemma2.1} we have $\cI_1\ll \tau^2Z$. Furthermore, by applying \cite[Lemma 3]{P:02lfp} and the trivial bound $|z_1z_1|\le |z_1|^2+|z_2|^2$ together with Lemma \ref{lemma2.3}, we obtain
$$\cI_2 \ll X^2 \sup_{\alpha \in \fm} |h_1(\alpha)h_2(\alpha)|^{1/2} + X^3 T(X)^{-1}.$$
We therefore deduce from (\ref{10.1}) and (\ref{10.2}) that
$$\tau XZ\ll Z^{1/2}X^{3/2}L(X)^{-1/4},$$
 and hence that $Z\ll \tau^{-2}XL(X)^{-1/2}$. Theorem \ref{primes} now follows on summing over dyadic intervals.

\section{Appendix: a variant of Br\"udern's pruning lemma}\label{sect:11}
It is convenient to have available a sharp version of Br\"udern's pruning lemma which avoids the loss of $\ep$-powers of the basic parameter. Although technically a straightforward modification of \cite[Lemma 2]{Bru:88}, we supply details here in order to provide a complete exposition. Our argument is motivated by the proof of \cite[Lemma 3.3]{BKW1}.\par

We begin with some notation. Let $k$ be a natural number with $k\ge 2$, let $N$ be a large real number, and put $P=N^{1/k}$. We suppose that $\cA\subseteq [1,P]\cap \Z$, and we write
$$F(\alpha)=\sum_{x\in \cA}e(\alpha x^k).$$
Finally, we define the multiplicative function $w_k(q)$ by taking
$$w_k(p^{uk+v})=\begin{cases} kp^{-u-1/2},&\text{when $u\ge 0$ and $v=1$,}\\
p^{-u-1},&\text{when $u\ge 0$ and $2\le v\le k$.}\end{cases}$$
Then according to \cite[Lemma 3]{Vau1986c}, whenever $a\in \Z$ and $q\in \NN$ satisfy $(a,q)=1$, the exponential sum
$$S_k(q,a)=\sum_{r=1}^qe(ar^k/q)$$
satisfies the bound $q^{-1}S_k(q,a)\ll w_k(q)$.

\begin{lemma}\label{lemma11.1}
Let $Q$ be a real number with $Q\le P$. When $a\in \Z$ and $q\in \NN$ satisfy $1\le a\le q \le Q$ and $(a,q)=1$, let $\cM(q,a)$ denote an interval contained in $[a/q-\tfrac{1}{2},a/q+\tfrac{1}{2}]$, and assume that the sets $\cM(q,a)$ are disjoint. Write $\cM$ for the union of the sets $\cM(q,a)$. Also, let $\gamma$ be a positive number, and let $G:\cM \rightarrow \C$ be a function satisfying
$$G(\alpha)\ll (q+N|q\alpha -a|)^{-1-\gamma}\quad \text{for}\quad \alpha \in \cM(q,a).$$
Then
$$\int_\cM G(\alpha)|F(\alpha)|^2\d\alpha \ll_\gamma P^2N^{-1}.$$
\end{lemma}

\begin{proof}One has
\begin{equation}\label{11.1}
\int_\cM G(\alpha)|F(\alpha)|^2\d\alpha \ll \sum_{1\le q\le Q}q^{-1-\gamma}\int_{-\infty}^\infty \sum^q_{\substack{a=1\\ (a,q)=1}}\frac{|F(a/q+\beta)|^2}{(1+N|\beta|)^{1+\gamma}}\d\beta .
\end{equation}
Write $c_q(h)$ for Ramanujan's sum, which we define by
$$c_q(h)=\sum^q_{\substack{a=1\\ (a,q)=1}}e(ah/q).$$
Then it follows that
$$\sum^q_{\substack{a=1\\ (a,q)=1}}|F(a/q+\beta)|^2=\sum_{x,y\in \cA}c_q(x^k-y^k)e(\beta (x^k-y^k)).$$
The estimate $|c_q(h)|\le (q,h)$ therefore conveys us from (\ref{11.1}) to the bound
\begin{equation}\label{11.2}
\int_\cM G(\alpha)|F(\alpha)|^2\d\alpha \ll_\gamma N^{-1}\sum_{1\le q\le Q}q^{-1-\gamma}\sum_{1\le x,y\le P}(q,x^k-y^k).
\end{equation}

\par Write $\rho(d)$ for the number of solutions of the congruence $x^k\equiv y^k\mmod{d}$ with $1\le x,y\le d$. Then, by sorting $x$ and $y$ into residue classes modulo $d$, we find that whenever $q\le P$ one has
\begin{align}
\sum_{1\le x,y\le P}(q,x^k-y^k)&\le \sum_{d|q}d\cdot \text{card}\{ 1\le x,y\le P:x^k\equiv y^k\mmod{d}\}\notag \\
&\le \sum_{d|q}(P/d+1)^2d\rho(d)\ll P^2\sum_{d|q}\rho(d)/d.\label{11.3}
\end{align}
But $\rho(d)$ is a multiplicative function of $d$, and when $d$ is a prime power $p^h$ with $h\ge 1$, one has
\begin{align*}
\rho(p^h)&=p^{-h}\sum_{b=1}^{p^h}|S_k(p^h,b)|^2=p^{-h}\sum_{l=0}^h\sum^{p^l}_{\substack{c=1\\ (c,p)=1}}(p^{h-l}|S_k(p^l,c)|)^2\\
&\le p^h\sum_{l=0}^hp^l w_k(p^l)^2\le k^2(h+1)p^{2h}w_k(p^h)^2.
\end{align*}
Since $\rho(d)$ is a multiplicative function of $d$, then $\sum_{d|q}\rho(d)/d$ must be a multiplicative function of $q$. When $q$ is a prime power $p^m$ with $m\ge 1$, therefore, one finds that
$$\sum_{d|q}\rho(d)/d\le \sum_{h=0}^m k^2(h+1)p^hw_k(p^h)^2\le k^4(m+1)^2p^mw_k(p^m)^2.$$
Consequently, we deduce from (\ref{11.3}) that
\begin{align}
q^{-1-\gamma}\sum_{1\le x,y\le P}(q,x^k-y^k)&\ll P^2\prod_{\substack{p^m\|q\\ m\ge 1}}\left( (p^m)^{-1-\gamma}k^4(m+1)^2p^mw_k(p^m)^2\right) \notag \\
&\le P^2\prod_{\substack{p^m\|q\\ m\ge 1}}\left( k^6(m+1)^2p^{-1-m\gamma}\right) .\label{11.4}
\end{align}

\par Next, on substituting (\ref{11.4}) into (\ref{11.2}), we deduce that
\begin{align*}
\int_\cM G(\alpha)|F(\alpha)|^2\d\alpha &\ll_\gamma P^2N^{-1}\sum_{q=1}^\infty \prod_{\substack{p^m\|q\\ m\ge 1}}\left( k^6(m+1)^2p^{-1-m\gamma}\right) \\
&=P^2N^{-1}\prod_p\Bigl( 1+p^{-1}\sum_{m=1}^\infty k^6(m+1)^2p^{-m\gamma}\Bigr)\\
&\le P^2N^{-1}\prod_p(1+Ap^{-1-\gamma}),
\end{align*}
for some positive number $A$ depending at most on $k$. Thus we may conclude that
\begin{align*}
\int_\cM G(\alpha)|F(\alpha)|^2\d\alpha &\ll_\gamma P^2N^{-1}\prod_p(1-p^{-1-\gamma})^{-A}\\
&\le P^2N^{-1}\zeta(1+\gamma)^A\ll_\gamma P^2N^{-1}.
\end{align*}
This completes the proof of the lemma.
\end{proof}

%\bibliographystyle{amsplain}
%\bibliography{fullrefs}

\begin{thebibliography}{10}

\bibitem{Bak:DI}
R.~C. Baker, \emph{{D}iophantine inequalities}, London Mathematical Society Monographs, New Series, vol. 1, Oxford University Press, Oxford, 1986.

\bibitem{BG:99}
V.~Bentkus and F.~G{\"o}tze, \emph{Lattice point problems and distribution of values of quadratic forms}, Annals of Math. (2) \textbf{150} (1999), 977--1027.

\bibitem{Bok:93}
K.~D. Boklan, \emph{A reduction technique in Waring's problem, I}, Acta Arith. \textbf{65} (1993), 147--161.

\bibitem{Bok:94}
K.~D. Boklan, \emph{The asymptotic formula in {W}aring's problem}, Mathematika \textbf{41} (1994), 329--347.

\bibitem{Bru:88}
J. Br\"udern, \emph{A problem in additive number theory}, Math. Proc. Cambridge Philos. Soc. \textbf{103} (1988), 27--33.

\bibitem{Bru:91}
J.~Br{\"u}dern, \emph{On {W}aring's problem for cubes}, Math.\ Proc.\ Cambridge Philos.\ Soc.\ \textbf{109} (1991), 229--256.

\bibitem{BCP:97}
J.~Br{\"u}dern, R.~J.~Cook and A.~Perelli, \emph{The values of binary linear forms at prime arguments}, Sieve methods, exponential sums, and their applications in number theory (Cardiff, 1995), London Math. Soc. Lecture Note Ser. 237, Cambridge University Press, Cambridge, 1997, pp.~87--100.

\bibitem{BKW1}
J.~Br{\"u}dern, K.~Kawada and T.~D. Wooley, \emph{Additive representation in thin sequences, {I}: {W}aring's problem for cubes}, Ann.\ Sci.\ {\'E}cole Norm.\ Sup. (4) \textbf{34} (2001), 471--501.

\bibitem{BKW3}
J.~Br{\"u}dern, K.~Kawada and T.~D. Wooley, \emph{Additive representation in thin sequences, {III}:
asymptotic formulae}, Acta Arith. \textbf{100} (2001), 267--289.

\bibitem{BKW8}
J.~Br{\"u}dern, K.~Kawada and T.~D. Wooley, \emph{Additive representation in thin sequences, {VIII}: {D}iophantine inequalities in review}, Number Theory. Dreaming in Dreams, Proceedings of the 5th China-Japan Seminar, Higashi-Osaka, 2008 (T.~Aoki, {\it et al.}, eds.), World Scientific, 2009, pp.~20--79.

\bibitem{BKW8add}
J. Br\"udern, K. Kawada and T. D. Wooley, \emph{Annexe to the gallery: an addendum to ``Additive representation in thin sequences, VIII: Diophantine inequalities in review''}, submitted, 6pp.

\bibitem{BW:00-27}
J.~Br{\"u}dern and T.~D. Wooley, \emph{On {W}aring's problem: two cubes and seven biquadrates}, Tsukuba J. Math. \textbf{24} (2000), 387--417.

\bibitem{BW:01}
J.~Br{\"u}dern and T.~D. Wooley, \emph{On Waring's problem for cubes and smooth Weyl sums}, Proc. London Math. Soc. (3) \textbf{82} (2001), 89--109.

\bibitem{BW:01a}
J.~Br{\"u}dern and T.~D. Wooley, \emph{On Waring's problem: three cubes and a sixth power}, Nagoya Math. J. \textbf{163} (2001), 13--53.

\bibitem{BW:04short}
J.~Br{\"u}dern and T.~D. Wooley, \emph{Additive representation in short intervals, {I}: {W}aring's problem for cubes}, Compositio Math. \textbf{140} (2004), 1197--1220.

\bibitem{Dav:42}
H.~Davenport, \emph{On sums of positive integral $k$th powers}, Amer. J. Math. \textbf{64} (1942), 189--198.

\bibitem{DH:46}
H.~Davenport and H.~Heilbronn, \emph{On indefinite quadratic forms in five variables}, J. London Math. Soc. \textbf{21} (1946), 185--193.

\bibitem{Ford:95}
K.~B. Ford, \emph{New estimates for mean values of {W}eyl sums}, Internat. Math. Res. Notices (1995), no. 3, 155--171.

\bibitem{Freem:alb}
D.~E. Freeman, \emph{Asymptotic lower bounds for Diophantine inequalities}, Mathematika \textbf{47} (2000), 127--159.

\bibitem{Freem:af}
D.~E. Freeman, \emph{Asymptotic lower bounds and formulas for Diophantine inequalities}, Number Theory for the Millennium (Urbana, IL, 2000) (M.~A.~Bennett et. al., ed.), vol.~2, 2002, pp.~57--74.

\bibitem{HT:86}
R. Hall and G. Tenenbaum, \emph{Divisors}, Cambridge University Press, Cambridge, 1988.

\bibitem{HB:88}
D.~R. Heath-Brown, \emph{Weyl's inequality, {H}ua's inequality, and {W}aring's problem}, J. London Math. Soc. (2) \textbf{38} (1988), 216--230.

\bibitem{Hoo:79}
C. Hooley, \emph{On a new technique and its applications to the theory of numbers}, Proc. London Math. Soc. (3) \textbf{38} (1979), 115--151.

\bibitem{Kaw:96}
K. Kawada, \emph{On the sum of four cubes}, Mathematika \textbf{43} (1996), 323--348.

\bibitem{KW:10rel}
K.~Kawada and T.~D. Wooley, \emph{Relations between exceptional sets for additive problems}, J. London Math.\ Soc.\ (2) \textbf{82} (2010), 437--458.

\bibitem{KW:10dav}
K.~Kawada and T.~D. Wooley, \emph{Davenport's method and slim exceptional sets: the asymptotic formulae in {W}aring's problem}, Mathematika \textbf{56} (2010), 305--321.

\bibitem{Kum:05}
A.~V. Kumchev, \emph{On the {W}aring-{G}oldbach problem: exceptional sets for sums of cubes and higher powers}, Canad.\ J.\ Math.\ \textbf{57} (2005), 298--327.

\bibitem{P:02lfp}
S.~T. Parsell, \emph{Irrational linear forms in prime variables}, J. Number Theory \textbf{97} (2002), 144--156.

\bibitem{P:ineq3}
S.~T. Parsell, \emph{On simultaneous diagonal inequalities, {III}}, Quart. J. Math. \textbf{53} (2002), 347--363.

\bibitem{PW:02}
S.~T. Parsell and T.~D. Wooley, \emph{On pairs of diagonal quintic forms}, Compositio Math. \textbf{131} (2002), 61--96.

\bibitem{Va:86c}
R.~C. Vaughan, \emph{On {W}aring's problem for cubes}, J. Reine Angew. Math. \textbf{365} (1986), 122--170.

\bibitem{Vau1986c}
R.~C. Vaughan, \emph{On {W}aring's problem for smaller exponents}, Proc. London Math. Soc. (3)
\textbf{52} (1986), 445--463.

\bibitem{Va:86se}
R.~C. Vaughan, \emph{On {W}aring's problem for smaller exponents, {II}}, Mathematika \textbf{33} (1986), 6--22.

\bibitem{Va:89}
R.~C. Vaughan, \emph{A new iterative method in Waring's problem}, Acta Math. \textbf{162} (1989), 1--71.

\bibitem{V:HL}
R.~C. Vaughan, \emph{The {H}ardy-{L}ittlewood method}, 2nd ed., Cambridge University Press, Cambridge, 1997.

\bibitem{VW:91}
R.~C. Vaughan and T.~D. Wooley, \emph{On {W}aring's problem: some refinements}, Proc. London Math. Soc. (3) \textbf{63} (1991), 35--68.

\bibitem{VW:95}
R.~C. Vaughan and T.~D. Wooley, \emph{Further improvements in {W}aring's problem}, Acta Math. \textbf{174} (1995), 147--240.

\bibitem{VW:00}
R.~C. Vaughan and T.~D. Wooley, \emph{Further improvements in {W}aring's problem, {IV}: higher
  powers}, Acta Arith. \textbf{94} (2000), 203--285.

\bibitem{W:92war}
T.~D. Wooley, \emph{Large improvements in {W}aring's problem}, Annals of Math. (2) \textbf{135} (1992), 131--164.

\bibitem{W:92vin}
T.~D. Wooley, \emph{On {V}inogradov's mean value theorem}, Mathematika \textbf{39}
  (1992), 379--399.

\bibitem{W:95sws}
T.~D. Wooley, \emph{New estimates for smooth Weyl sums}, J. London Math. Soc. (2) \textbf{51} (1995), 1--13.

\bibitem{W:95bcc}
T.~D. Wooley, \emph{Breaking classical convexity in {W}aring's problem: sums of cubes and quasi-diagonal behaviour}, Invent. Math. \textbf{122} (1995), 421--451.

\bibitem{W:3cubes}
T.~D. Wooley, \emph{Sums of three cubes}, Mathematika \textbf{47} (2000), 53--61.

\bibitem{W:02slimcube}
T.~D. Wooley, \emph{Slim exceptional sets for sums of cubes}, Canad. J. Math. \textbf{54} (2002), 417--448.

\bibitem{W:02slimsquare}
T.~D. Wooley, \emph{Slim exceptional sets for sums of four squares}, Proc.\ London Math.\ Soc.\ (3) \textbf{85} (2002), 1--21.

\bibitem{W:FAF}
T.~D. Wooley, \emph{On Diophantine inequalities: {F}reeman's asymptotic formulae}, Proceedings of the session in analytic number theory and Diophantine equations (Bonn, January--June, 2002) (D.~R. Heath-Brown and B.~Z. Moroz, eds.), no. 360, Bonner Mathematische Schriften, 2003.

\bibitem{W:03slim}
T.~D. Wooley, \emph{Slim exceptional sets and the asymptotic formula in {W}aring's problem}, Math. Proc. Cambridge Philos. Soc. \textbf{134} (2003), 193--206.

\bibitem{W:12war}
T.~D. Wooley, \emph{The asymptotic formula in {W}aring's problem}, Internat.\ Math.\ Res.\ Notices (2012), no. 7, 1485--1504.

\bibitem{W:12ec}
T.~D. Wooley, \emph{Vinogradov's mean value theorem via efficient congruencing}, Annals of Math. (2) \textbf{175} (2012), 1575--1627.

\bibitem{W:13ec2}
T.~D. Wooley, \emph{Vinogradov's mean value theorem via efficient congruencing, {II}}, Duke Math. J. (accepted, to appear).

\end{thebibliography}

\providecommand{\bysame}{\leavevmode\hbox
to3em{\hrulefill}\thinspace}
\providecommand{\MR}{\relax\ifhmode\unskip\space\fi MR }
% \MRhref is called by the amsart/book/proc definition of \MR.
\providecommand{\MRhref}[2]{%
  \href{http://www.ams.org/mathscinet-getitem?mr=#1}{#2}
} \providecommand{\href}[2]{#2}

\end{document}